\documentclass[14pt]{article}
\usepackage{amsmath}
\usepackage{amssymb}
\usepackage{mathrsfs}
\usepackage[all,cmtip]{xy}
\usepackage{amsthm,color}
\usepackage{tikz-cd}
\usepackage{authblk}

\let\margin\marginpar
\newcommand\myMargin[1]{\margin{\raggedright\scriptsize #1}}
\renewcommand{\marginpar}[1]{\myMargin{#1}}

\newtheorem{lemma}{Lemma}[section]
\newtheorem{theorem}[lemma]{Theorem}
\newtheorem{corollary}[lemma]{Corollary}
\newtheorem{conjecture}[lemma]{Conjecture}
\newtheorem{prop}[lemma]{Proposition}

\theoremstyle{definition}
\newtheorem{definition}[lemma]{Definition}

\newtheorem{remark}[lemma]{Remark}

\theoremstyle{remark}
\newtheorem*{proof*}{Proof}
\numberwithin{equation}{section}

\newcommand{\wt}{\widetilde}

\def\RsHom{{\RR{\underline{Hom}}}}
\def\gr{{\mathbf{gr}}}
\def\D{\mathrm{D}}

\def\coh{\mathrm{coh}}

\def\Vect{\ul{Vect}}
\def\Sym{\mathrm{Sym}}

\def\Ext{{\mathrm{Ext}}}

\def\Hom{{\mathrm{Hom}}}
\def\sHom{{\mathscr{H}om}}
\def\End{{\mathrm{End}}}

\def\sEnd{{\mathscr{E}nd}}

\def\Perf{{\underline{Perf}}}

\def\Cok{{\text{coker}}}
\def\bi{{\bf{1}}}

\def\p{{\prime}}
\def\Map{{\bf{Map}}}
\def\dVt{\underline{Vect}}

\def\dCx{\RR{\underline{Cplx}}}
\def\qgr{{\mathrm{qgr}}}

\newcommand{\id}{\operatorname{id}}

\def\PP{{\mathbb P}}

\def\ZZ{{\mathbb Z}}
\def\CC{{\mathbb C}}
\def\LL{{\mathbb L}}
\def\TT{{\mathbb T}}

\def\RR{{\mathbb R}}
\def\HH{{\mathbb H}}
\def\GG{{\mathbb G}}

\def\bK{{\bf{K}}}

\def\cF{{\cal{F}}}
\def\cE{{\cal{E}}}
\def\cO{{\cal{O}}}
\def\cC{{\cal{C}}}
\def\cL{{\cal{L}}}
\def\cI{{\cal{I}}}
\def\cM{{\cal{M}}}
\def\cN{{\cal{N}}}
\def\cA{{\cal{A}}}
\def\cB{{\cal{B}}}
\def\cH{{\cal{H}}}

\def\cT{{\cal{T}}}
\def\cU{{\cal{U}}}
\def\cV{{\cal{V}}}
\def\cW{{\cal{W}}}
\def\cR{{\cal{R}}}

\def\cQ{{\cal{Q}}}

\def\fp{{\mathfrak{p}}}

\def\ft{{\mathfrak{t}}}
\def\fg{{\mathfrak{g}}}
\def\fa{{\mathfrak{a}}}

\def\fm{{\mathfrak{m}}}

\def\lg{{\langle}}
\def\rg{{\rangle}}

\def\sC{{\mathscr{C}}}

\def\tr{{\mathrm{tr}}}

\def\rk{{\mathrm{rk}}}
\def\ad{{\mathrm{ad}}}
\def\Res{{\mathrm{Res}}}
\def\deg{{\mathrm{deg}}}

\def\ii{{\mathrm{i}}}

\def\im{{\mathrm{im}}}
\def\ker{{\mathrm{ker}}}

\def\mod{{~\mathrm{mod}~}}

\def\ep{{\epsilon}}

\def \hs {\hspace{.2in}}

\def \ot {\otimes}
\def \bt {\boxtimes}

\def \wh{\widehat}

\def\ul{\underline}
\def\gr{\mathrm{gr}}
\def \Z {\mathbb Z}
\def \mc {\ep{\mathscr{C}}^{\gr}}
\def \dg {\mathbf{dg}}

\title{Shifted Poisson structures and moduli spaces of complexes}
\date{}
\author[1]{Zheng Hua\thanks{huazheng@maths.hku.hk}}
\author[2]{Alexander Polishchuk \thanks{apolish@uoregon.edu}}
\affil[1]{Department of Mathematics, the University of Hong Kong, Hong Kong SAR, China}
\affil[2]{University of Oregon and National Research University Higher School of Economics}

\begin{document}
\maketitle

\begin{center}
\begin{abstract}
\vspace{1cm}
In this paper we study the moduli stack of complexes of vector bundles (with chain isomorphisms)
over a smooth projective variety $X$ via derived algebraic geometry.
We prove that if $X$ is a Calabi-Yau variety of dimension $d$ then this moduli stack has a $(1-d)$-shifted Poisson
structure.
In the case $d=1$, we construct a natural foliation of the moduli stack by $0$-shifted symplectic substacks. 
We show that our construction recovers various known Poisson structures associated to complex elliptic curves,  including the Poisson structure on Hilbert scheme of points on elliptic quantum projective planes studied by Nevins and Stafford, and the Poisson structures on the moduli spaces of stable triples over an elliptic curves considered by one of us.
We also relate the latter Poisson structures to the semi-classical limits of the 
elliptic Sklyanin algebras studied by Feigin and Odesskii. 
\end{abstract} 
\end{center}

\newpage
\section{Introduction}

The framework of derived algebraic geometry, including the theory of derived Artin stacks (see \cite{HAGII}, \cite{To09}), 
has become an important tool in understanding moduli spaces of vector bundles and their generalizations.
In this paper we study the moduli spaces of bounded complexes of vector bundles in this context.

Our main inspiration comes from the work of Pantev-To\"en-Vaqui\'e-Vezzosi \cite{PTVV}, where the classical
symplectic structure on the moduli spaces of sheaves over K3-surfaces discovered by Mukai is ``explained" and
generalized using {\it shifted symplectic structures}. Namely, they show that the derived moduli stack $\RR\Perf(X)$ of perfect complexes over a $d$-dimensional Calabi-Yau variety $X$ has a canonical $(2-d)$-shifted symplectic structure.

In this paper we consider the derived moduli stack $\dCx(X)$ of bounded complexes of vector bundles
over a smooth projective variety $X$. Note that here we consider bounded complexes of vector bundles with isomorphisms given
by chain isomorphisms (not quasi-isomorphisms). Thus, it is very different from the moduli stack of perfect complexes where isomorphisms in derived category are allowed. This also makes our moduli stack to be closely related to moduli spaces of Higgs bundles,
of stable pairs and triples (see \cite{BG96}) and of coherent systems (see \cite{LP95}). 
More precisely, $\dCx(X)$ is defined as the derived mapping stack from $X$ to the derived stack $\dCx$. The latter classifies 
perfect graded mixed complexes with fixed Tor-amplitude $[0,0]$, where we
use the formalism of \cite[Sec.\ 1]{CPTVV}. The construction can be found in Section \ref{sec_mod_cplx}.

 Our main result is
that if $X$ is a Calabi-Yau variety 
of dimension $d$ then $\dCx(X)$ admits a $(1-d)$-shifted Poisson structure (see Theorem \ref{Mainthm1}). 
Here we use the theory of shifted Poisson structures developed in \cite{CPTVV} (see also \cite{Mel16},
\cite{MelSaf16}, \cite{Prid17}, \cite{Sp16}).


Our original motivation was to ``explain" natural Poisson structures on moduli spaces of stable triples and stable pairs over
an elliptic curve $C$ (see \cite{Pol98}), as well as those on moduli spaces of complexes over an elliptic curve constructed by
Nevins-Stafford \cite{NS06} and by Li \cite{Li14}. We indeed show that the relevant moduli stacks admit $0$-shifted Poisson structures
which have the classical Poisson structures as their ``shadows" (see Theorems \ref{Thmapp1} and \ref{stable-triples-thm}).
Note that Safronov in \cite{Saf17} gives a similar construction of a $0$-shifted Poisson structure on the moduli stack of 
$P$-bundles over an elliptic curve, where $P$ is a parabolic subgroup in a simple algebraic group (the corresponding classical 
Poisson structure was discovered by Feigin-Odesskii in \cite{FO95}).

As an additional bonus, the derived geometry point of view clarifies the picture with the symplectic leaves of
these Poisson structures. Namely, we show that they can be recovered from the natural map
$$\dCx(C)\to \RR\Perf(C)\times \RR\dVt^\gr(C)$$
by considering the homotopy fibers over stacky points in the target (see Corollary \ref{leaves}). More precisely, we prove
that these homotopy fibers carry $0$-shifted symplectic structures. Under additional assumptions this gives a classical symplectic
structure on the coarse moduli spaces of such fibers. 
In Sections \ref{ell-def-sympl-leaves-sec} and \ref{triples-sec} we show how this approach can be used to
study symplectic leaves of the classical Poisson structures associated with an elliptic curve.

In a somewhat different direction, in Section \ref{FO-sec} we check that the Poisson structure on the projectivization of
the space $\Ext^1(L,\mathcal{O}_C)$, for a degree $n$ line bundle $L$ on $C$, arising from the identification with a
moduli space of stable pairs,
coincides with the semi-classical limit of the Feigin-Odesskii algebras $Q_{n,1}(C,\eta)$. 



One should note that the foundations of the theory of shifted Poisson structures are not quite complete at the moment.
Namely, the definition of an $n$-shifted Poisson structure used by Spaide \cite{Sp16}, in terms of a Lagrangian structure
on a morphism to an $(n+1)$-shifted symplectic formal derived stack, is different from the definition in \cite{CPTVV}.
The proof of the equivalence between the two definitions should appear in a forthcoming paper by K.~Costello and 
N.~Rozenblyum. However, one implication of this equivalence has been already proved by Melani and Safronov \cite{MelSaf16}:
they prove that an $n$-shifted Poisson structure in the sense of Spaide gives rise to an $n$-shifted Poisson structure
as defined in \cite{CPTVV}. 

Our method of constructing $(1-d)$-shifted Poisson structure on the moduli stack of complexes is via constructing
a Lagrangian structure on a morphism to a moduli stack equipped with an $(2-d)$-shifted symplectic structure.
Thus, by the results of \cite{Sp16} and \cite{MelSaf16}, this will give a shifted Poisson structure with respect to either definition.

\paragraph{Acknowledgments.} We are grateful to Tony Pantev, Ted Spaide, Jiang-Hua Lu, Kai Behrend and Pavel Safronov for many valuable comments. Particular thanks go to Jon Pridham for kindly explaining the construction in \cite{Prid12} to us.
The research of Z.H. is supported by RGC Early Career grant no. 27300214, GRF grant no. 17330316  and NSFC Science Fund for Young Scholars no. 11401501. A.P. is supported in part by the NSF grant DMS-1400390 and by the Russian Academic Excellence Project `5-100'.

\section{Moduli spaces of complexes and perfect complexes}
The main purpose of this section is to briefly recall the construction of the moduli stack of objects in a dg category, due to T\"oen and Vaqui\'e \cite{TV07}, and to set up some notation involving dg-categories and derived stacks that will be used later.

\subsection{Moduli of objects in dg-categories}\label{sec_Rperf}
Let $k$ be a field of characteristic zero. We denote by $\sC(k)$ the category of (unbounded) cochain complexes of $k$-modules. It carries a standard model structure (\cite[Definition 2.3.3]{Ho99}). For $L,M\in\sC(k)$ and $n\in\ZZ$, let $\Hom^n(L,M)$ denote the $k$-module of morphisms $f:L\to M$ of graded objects of degree $n$. Then
\[
\sC(k)(L,M):=\Hom^\bullet(L,M)=\bigoplus_{n\in\ZZ}\Hom^n(L,M)
\] with the differential induced by the differentials on $L$ and $M$. This makes $\sC(k)$ into a dg-category. The model structure and the dg-category structure on $\sC(k)$ both lead to the (same) homotopy category of $\sC(k)$, which carries a symmetric monoidal structure.

      Let $T$ be a small dg-category over $k$. A \emph{(left) dg-$T$-module} is a dg functor $T\to \sC(k)$. We denote the category of \emph{(left) dg-$T$-module} by $T-Mod$. By \cite[4.2.18]{Ho99}, the category $T-Mod$ can be endowed with a $\sC(k)$-model structure. Moreover, its model structure is compactly generated. For the details, we refer to \cite[Section 2.1]{TV07}. 

To any dg-algebra $B$ over $k$, we associate a dg-category, also denoted by $B$ with the unique 
object $\star$ such that the dg-algebra of endomorphisms of this object $B(\star,\star)$ is equal to $B$. 
We denote by $\bi$ the dg-category with the unique object $\star$ and $\bi(\star,\star)=k$. The dg-category $\bi-Mod$ is just $\sC(k)$. More generally, for a $k$-algebra $B$, viewed as a dg-algebra concentrated in degree $0$, the dg-category $B-Mod$ is  the dg-category of complexes of $B$-modules.

The category of (small) dg-categories over $k$, denoted by $dgcat_k$, has a model structure \cite{Tab05}. Its homotopy category $Ho(dgcat_k)$ has a natural symmetric monoidal structure, with the corresponding internal Hom denoted by $\RsHom$ and the mapping spaces denoted by $Map(-,-)$ (see \cite{Ho99} for the precise definition of the mapping spaces). We consider these
mapping spaces as objects in the homotopy category of simplicial sets $Ho(sSet)$. Let us denote the full subcategory of $T^{op}-Mod$ consisting of cofibrant objects by $\wh{T}$. The objects in $\wh{\bi}$ are precisely the projective dg-$k$-modules. In terms of the internal Hom, we have 
\begin{equation}\label{hat1}
\wh{T}=\RsHom(T^{op},\wh{\bi}).
\end{equation}
We denote by $\wh{T}_{pe}$ the full subcategory of perfect objects in $\wh{T}$ (see Definition 2.3 of \cite{TV07}). 
The Yoneda functor 
\[
\underline{h}: T\to \wh{T}
\] factors through the subcategory $\wh{T}_{pe}$ (see section 2.2 \cite{TV07}).

For a dg-category $T$, following \cite{TV07}, we consider the functor
\begin{align*}
\cM_T: &sk-CAlg \to sSet\\
&A \mapsto Map_{dgcat_k}(T^{op},\wh{A}_{pe})
\end{align*}
where $sk-CAlg$ is the category of simplicial commutative $k$-algebras, $sSet$ is the category of simplicial sets and $Map_{dgcat_k}(T^{op},\wh{A}_{pe})$ is the mapping space of the model category of dg-categories. And $\wh{A}_{pe}$ stands for $\wh{N(A)-Mod}_{pe}$ with $N$ being the normalization. We denote $k-CAlg$ for the category of commutative $k$-algebras.
By \cite[Lemma 3.1]{TV07}, 
$\cM_T$ has a structure of $D^-$-stack in the sense of \cite[Section 2.2]{HAGII}.

The main theorem of \cite{TV07} is the following.
\begin{theorem}(Theorem 3.6, Corollary 3.17 \cite{TV07})\label{MthmTV}
Let $T$ be a dg-category of finite type (e.g. $T$ is saturated). Then the $D^-$-stack $\cM_T$ is locally geometric and locally of finite presentation. Moreover, for any pseudo-perfect $T^{op}$-dg-module $E$, corresponding to a global point of $\cM_T$, the tangent complex of $\cM_T$ at $E$ is given by
\[
\TT_{\cM_T,E}\simeq \RsHom(E,E)[1].
\]
In particular, if $E$ is quasi-representable by an object $x$ in $T$, then we have
\[
\TT_{\cM_T,E}\simeq T(x,x)[1].
\]
\end{theorem}
\begin{remark}
The definition of a dg-category of {\it finite type} can be found in \cite[Definition 2.4]{TV07}.
The definitions of \emph{pseudo-perfect object}, \emph{locally geometric}, \emph{locally of finite presentation} and \emph{quasi-representable} can be found in 
\cite[Definition 2.7, Section 2.3, Definition 2.4]{TV07}. The locally geometric condition guarantees the existence of the cotangent complexes (by \cite[Corollary 2.2.3.3]{HAGII}).
\end{remark}

When $T=\bi$, $\cM_\bi$ classifies the perfect complexes of $k$-modules up to quasi-isomorphisms. We denote $\cM_\bi$ by $\RR\Perf$. Let $a,b\in\ZZ$ be two integers with $a\leq b$. We define $\cM_{\bi}^{[a,b]}\subset \cM_\bi$ to be the full sub-$D^-$-stack consisting of perfect complexes of $k$-modules of Tor amplitude contained in $[a,b]$ (\cite[Proposition 2.2]{TV07}). We find that
\[
\cM_\bi=\bigcup_{a\leq b}\cM^{[a,b]}_\bi,
\]
and $\cM^{[a,b]}_\bi$ is $n$-geometric and locally of finite presentation, for $n=b-a+1$. When $a=b$, $\cM^{[a,a]}_\bi$ is equivalent to the $D^-$-stack of vector bundles 
\begin{equation}\label{vect=BG}
\Vect\simeq \bigsqcup_{n} BGL_n.
\end{equation}

Let $V=\{\ldots\to V^i\to V^{i+1}\to\ldots\}$ be a perfect complex of $k$-modules. 
 We denote  $\dVt^\gr$ for the classifying stack of the underlying graded objects $\bigoplus_{i}V^i[-i]$.
Note that $\dVt^\gr=\bigcup_{n>0}\dVt^\gr_n$, where $\dVt^\gr_n$ is the $n$-fold product of $\dVt$. There is an obvious morphism $s:\dVt^{gr}\to \RR\Perf$, mapping $\bigoplus_{i}V^i[-i]$ to the corresponding perfect complex with zero differentials.

Let us denote by $D^-St(k)$ the homotopy category of $D^-$-stacks, and by $St(k)$ the category of (higher) stacks.
Considering a commutative $k$-algebra as a constant simplicial commutative $k$-algebra 
gives an embedding $k-CAlg\to sk-CAlg$. This embedding induces a truncation functor
\[
t_0: D^-St(k)\to St(k),
\] with the left adjoint 
\[
i: St(k)\to D^-St(k)
\] which is fully faithful.

One has natural isomorphisms
\[
\pi_1(\cM_T,E)\cong Aut(E), ~~~ \pi_i(\cM_T,E)\cong \Ext^{1-i}(E,E),
\] where $Aut$ and $\Ext$ are computed in the triangulated category $Ho(T^{op}-Mod)$.

The construction $T\mapsto \cM_T$ is contravariant in $T$, and gives rise to a contravariant functor from the model category of small dg-categories to the model category of $D^-$-stacks (see \cite[Section 3.1]{TV07}).  Moreover, it induces a functor between homotopy categories
\[
\cM_{-}: Ho(dgcat_k)^{op}\to D^-St(k).
\]
Let $F$ be a locally geometric $D^-$-stack. The natural adjunction map $i(t_0(F))\to F$, from the truncated stack to the derived stack, induces a map on tangent complexes and cotangent complexes
\[
\TT_{i(t_0(F))}\cong \tau^{\leq 0} \TT_F\to \TT_F,~~\LL_F\to \LL_{i(t_0(F))}\cong \tau^{\geq 0} \LL_F
\]
where $\tau^{\geq 0}$ and $\tau^{\leq 0}$ are the (smart) truncation functors of complexes. To simplify the notation, we will omit the functor $i$ and just write this map as $t_0(F)\to F$. For $A\in k-CAlg$ the set of $A$-points of the truncated stack $t_0(F)$ is given by 
$$t_0(F)(A)=F(A).$$

\begin{definition}\label{stacky-pt-def}
Fix a $T^{op}$-module $E$. We define $x_E$, the \emph{stacky} point representing $E$, as
 the subfunctor of $t_0(\cM_T)$ corresponding to the subcategory of $T$ consisting of objects isomorphic to $E$.
\end{definition}

There is a monomorphism of stacks
\[
x_E\to t_0(\cM_T)
\] whose tangent map is
\[
 \tau^{\leq 0} (\RsHom(E,E))[1]\to  \tau^{\leq 1} (\RsHom(E,E))[1].
\]
The induced morphism on cohomology is an isomorphism for negative degree and is zero for degree zero.

\subsection{Moduli space of complexes}\label{sec_mod_cplx}

The goal of this subsection is to construct a geometric $D^-$ stack $\dCx$ classifying bounded complexes of vector bundles up to chain isomorphisms,  and a morphism from $\dCx$ to $\RR\Perf$. The main idea is to view objects of $\dCx$ as graded mixed complexes with fixed Tor-amplitude.

First we briefly recall the basics of the category of graded mixed complexes following Section 1.1 of \cite{CPTVV}.
Let $\mc(k)$
denote the category of graded mixed complexes of $k$-dg-modules. Its objects consist of families of dg-$k$-modules $E:=\{E(p)\}_{p\in\ZZ}$, together with families of morphisms 
\[
\ep: E(p)\to E(p+1)[1],
\] such that $\ep^2=0$. In order to avoid confusion, we call $p$ the \emph{weight degree} to distinguish it from the cohomological grading on $E(p)$.

For $E,F\in \mc(k)$, the Hom space $\underline{Hom}^\gr_{\mc(k)}(E,F)$ in $\mc(k)$ has its weight $p$ piece 
\[
\underline{Hom}^\gr_{\mc(k)}(E,F)(p):=\prod_{q\in \ZZ}\underline{Hom}_{\sC(k)}(E(q),F(q+p))
\] for any $p\in\ZZ$. The mixed differential $\ep_p: \underline{Hom}^\gr_{\mc(k)}(p)\to \underline{Hom}^\gr_{\mc(k)}(p+1)[1]$ has its $q$-component defined by
\[
\alpha+\beta: \prod_{q^\p\in\ZZ}\underline{Hom}_{\sC(k)}(E(q^\p), F(q^\p+p))\to \underline{Hom}_{\sC(k)}(E(q),F(p+q+1))[1]
\] where 
\[
\xymatrix{
\prod_{q^\p\in\ZZ}\underline{Hom}_{\sC(k)}(E(q^\p),F(q^\p+p))\ar[r]^{pr}\ar[rd]_\alpha & \underline{Hom}_{\sC(k)}(E(q),F(q+p))\ar[d]^{\ep_F\cdot}\\
& \underline{Hom}_{\sC(k)}(E(q),F(p+q+1))[1]
}
\] and
\[
\xymatrix{
\prod_{q^\p\in\ZZ}\underline{Hom}_{\sC(k)}(E(q^\p),F(q^\p+p))\ar[r]^{pr}\ar[rd]_\beta & \underline{Hom}_{\sC(k)}(E(q+1),F(q+1+p))\ar[d]^{\cdot \ep_E}\\
& \underline{Hom}_{\sC(k)}(E(q),F(p+q+1))[1]
}
\] commute.

The category $\sC^\gr(k):=\prod_{p\in\ZZ}\sC(k)$ is naturally a symmetric monoidal model category with model structure defined component-wise, and a monoidal structure defined by
\[
(E\otimes E^\p)(p):=\bigoplus_{i+j=p} E(i)\otimes E^\p(j).
\]
One can check that $\mc(k)$ is equipped with a symmetric monoidal category structure, defined through the forgetful functor 
\[
\mc(k)\to \sC^\gr(k)
\] that forgets the mixed structure.

We denote the $\infty$-categories corresponding to $\sC(k)$, $\sC(k)^\gr$ and $\mc(k)$ by $\dg_k$, $\dg^\gr_k$ and $\ep\dg^\gr_k$. We denote the $\infty$-functor corresponding to the forgetful functor by
\[
U_\ep: \ep\dg^\gr_k \to \dg^\gr_k.
\]
On the other hand, there is a second $\infty$-functor, called the \emph{realization functor}:
\[
|-|: \ep\dg^\gr_k\to \dg_k
\] defined on the strict model by 
\[
E\mapsto\prod_{p\geq 0} E(p),
\]
where the right hand side is endowed with the total differential $=$ sum of the cohomological differential and the mixed differential.
\begin{remark}
We provide an alternative view to objects in $\mc(k)$ which is sometime useful. Given an object $E:=\{ E(p)\}_{p\in\ZZ}$ in $\mc(k)$, we consider a double complex $E^{\bullet,\bullet}$ whose $p$-th column is $E(p)$ with internal (cohomological) differential, and horizontal differential given by the mixed structure. The forgetful functor forgets the horizontal differential on $E^{\bullet,\bullet}$ and the realization functor maps $E^{\bullet,\bullet}$ to the total complex.
\end{remark}
Let $T$ be a small dg-category over $k$. A (left) $\ep$-dg-$T$-module is a dg-functor $T\to \mc(k)$. Denote the category of $\ep$-dg-$T^{op}$-modules by  $\ep T^{op}-Mod$. It is endowed with a $\mc(k)$-model structure. Denote the full subcategory of $\ep T^{op}-Mod$ consisting of cofibrant objects by $\ep \wh{T}$, and the full subcategory of $\ep\wh{T}$ consisting of perfect objects by $\ep\wh{T}_{pe}$. Replacing the target model category $\mc(k)$ by $\sC^\gr(k)$, we can define similarly the category of graded dg-$T^{op}$-modules $\gr T^{op}-Mod$, the full subcategory of cofibrant objects $\gr\wh{T}$ and the full subcategory of perfect objects $\gr\wh{T}_{pe}$.  

Consider the moduli functors:
\[
\gr\cM_T: sk-CAlg\to sSet
\]
\[
A\mapsto Map_{dgcat_k}(T^{op},\gr\wh{A}_{pe})
\]
and
\[
\ep\cM_T: sk-CAlg\to sSet
\]
\[
A\mapsto Map_{dgcat_k}(T^{op},\ep\wh{A}_{pe})
\] where $\ep\wh{A}_{pe}$ (resp. $\gr\wh{A}_{pe}$) is the dg-category of perfect graded mixed complexes of dg-$N(A)$-modules (resp. perfect graded dg-$N(A)$-modules). It follows from the proof of Lemma 3.1 of \cite{TV07} that $\gr\cM_T$ and $\ep\cM_T$ are $D^-$-stacks.

The forgetful functor and the realization functor induces natural transformations of moduli functors, and therefore morphisms of $D^-$-stacks. By an abuse of notations, we denote the natural transformations by
\[
\xymatrix{
& \ep\cM_T\ar[rd]^{|-|}\ar[ld]_{U_\ep} &\\
\gr\cM_T & & \cM_T
}
\]

When $T=\bi$, $\ep\cM_\bi$ (resp. $\gr\cM_\bi$) classifies the perfect graded mixed complexes of dg-$k$-modules (resp. perfect graded dg-$k$-modules) up to quasi-isomorphisms of dg-$k$-modules. Let $a,b\in\ZZ$ be two integers with $a\leq b$. Define $\ep\cM_\bi^{[a,b]}\subset \ep\cM_\bi$ (resp. $\gr\cM_\bi^{[a,b]}\subset \gr\cM_\bi$) to be the full substack consisting of perfect mixed complexes of dg-$k$-modules (resp. perfect graded dg-$k$-modules) of Tor amplitude contained in $[a,b]$. 

We denote $\ep\cM_\bi^{[0,0]}$ by $\dCx$. Clearly, $\gr\cM_\bi^{[0,0]}$ is simply the stack of graded vector bundles $\Vect^{gr}$, which is  a $1$-geometric stack and locally of finite presentation. We refer the readers to Section 2.3 of \cite{TV07} for the definitions of $n$-geometric and locally finite presentation. A $D^-$-stack $F$ is called a \emph{derived Artin stack} if $F$ is $n$-geometric for some $n\in\ZZ$.

Now we study the stack $\dCx$ by describing its $A$-points.  For $A\in k-CAlg$, the $A$-points $\dCx(A)$ classify bounded graded mixed complexes of finitely generated projective $A$-modules, i.e.  bounded complexes (with respect to the weight decomposition) of finitely generated projective $A$-modules, or equivalently a graded $A[\ep]/(\ep^2)$-module that is a finitely generated projective $A$-module. 
In particular, the $k$-points of $\dCx$ correspond to bounded complexes of finite dimensional $k$-vector spaces:
\[
\xymatrix{
V:=\{\ldots\ar[r] &V(p-1)\ar[r]^\ep & V(p)\ar[r]^\ep & V(p+1)\ar[r] & \ldots\}
}
\]
Because the model structure on $\mc(k)$ is pull back from that on $\sC^\gr(k)$, two objects in $\dCx$ are equivalent if and only if they are chain isomorphic.
\begin{prop}
The $D^-$-stack $\dCx$ is $1$-geometric and locally of finite presentation. 
\end{prop}
\begin{proof}
By Theorem 2.2.6.12 of \cite{HAGII} or by Theorem 4.12 of \cite{Prid12} which are both special cases of the representability criteria of Lurie \cite{Lurie}, it suffices to  check that
\begin{enumerate}
\item[$(1)$] $t_0(\dCx)$ is a (underived) Artin stack locally of finite presentation;
\item[$(2)$] $\dCx$ admits an obstruction theory; 
\item[$(3)$] $\dCx$ is nilcomplete, i.e. it commutes with the homotopy limit given by the Postnikov tower of $A\in sk-CAlg$.
\end{enumerate} 
If $(1)$ holds, then by Proposition 1.32 and 1.33 of \cite{PridRep} to verify conditions $(2)$ and $(3)$ for $\dCx$ it suffices to check the following condition:
\begin{enumerate}
\item[$(4)$] For $A\in k-CAlg$ and $E\in \dCx(A)$, the cohomology groups of the tangent complex at $E$ are finitely generated $A$-modules.
\end{enumerate}
Because an $A$-point $E$ corresponds to a graded $A[\ep]/(\ep^2)$-module that is a finitely generated projective $A$-module, the $i$-th cohomology group of the tangent complex at $E$ is equal to $\Ext^{i+1}_{A[\ep]/(\ep^2)}(E,E)$ which is a finitely generated $A$-module.

Fix integers $p,q$ such that $p\leq q$.
Let $v_{(p,q)}=(v_p,v_{p+1},\ldots,v_q)$ be a vector of positive integers. Denote $\dCx_{v_{(p,q)}}$ for the substack whose $A$-points are complexes (with respect to the weight degree) 
\[
0\to V(p)\to V(p+1)\to\ldots \to V(q)\to 0
\] of projective $A$-modules such that $\rk~V(i)=v_i$. To prove $(1)$, it suffices to show that $t_0(\dCx_{v_{(p,q)}})$ is an Artin stack of finite presentation. We will prove it by representing $t_0(\dCx_{v_{(p,q)}})$ as (underived) fiber product of Artin stacks. Denote $\cF_{v_{(p,q)}}$ for the (underived) stack of graded vector bundles with  rank vector $v_{(p,q)}$. Denote $\cV_{v_{(p,q)}}$ and $\cW_{v_{(p,q)}}$ for the vector bundle stacks over $\cF_{v_{(p,q)}}$ with fiber being the universal bundle $\bigoplus_{i=p}^{q-1} \Hom(V(i),V(i+1))$ and $\bigoplus_{i=p}^{q-2} \Hom(V(i),V(i+2))$ respectively. Let $\mu: \cV_{v_{(p,q)}}\to \cW_{v_{(p,q)}}$ be the morphism of stacks defined by composition and $\iota : \cF_{v_{(p,q)}}\to \cW_{v_{(p,q)}}$ be the morphism defined by zero section. Then $t_0(\dCx_{v_{(p,q)}})$ is the fiber product of the diagram of morphisms of Artin stacks  of finite presentation:
\[
\xymatrix{
t_0(\dCx_{v_{(p,q)}})\ar[r]\ar[d] &  \cV_{v_{(p,q)}}\ar[d]^\mu\\
 \cF_{v_{(p,q)}}\ar[r]^\iota  & \cW_{v_{(p,q)}}
}
\]

\end{proof}

The tangent complex of $\dCx(k)$ at $V$ can be computed by the reduced Hochschild cochain complex of $k[\ep]/(\ep^2)$ with coefficients in the bimodule $\End_\bullet(V):=\bigoplus_{p\in\ZZ}\bigoplus_{q\in \ZZ}\Hom_k(V(p),V(p+q))$:
\[
\bigoplus_{n\geq 0}\Hom_\gr(\fm^{\otimes n},\End(V))
\] where $\fm$ is the maximal ideal of $k[\ep]/(\ep^2)$. Here $\Hom_\gr$ means homogeneous map of degree zero. It is easy to check that the Hochschild cochain complex is precisely 
\[
\End_{\geq 0}(V):=\bigoplus_{p\in\ZZ}\bigoplus_{q\geq 0}\Hom_k(V(p),V(p+q))
\] with differential given by $\ep\cdot$.
Denote $\End_0(V)$ for $\bigoplus_{p\in\ZZ}\Hom_k(V(p),V(p))$.

By setting $T=\bi$ and restricting the morphisms $U_\ep$ and $|-|$ to the substack $\dCx$, we get a diagram of morphisms of $D^-$-stacks: 
\[
\xymatrix{
& \dCx\ar[rd]^{q}\ar[ld]_{p} &\\
\Vect^\gr & & \RR\Perf
}
\]
For $A \in sk - CAlg$ and a simplicial resolution $A^{sim}$ of $A$, the tangent maps of $q$ and $p$ at an $A$-point $V\in\dCx(A)$ can be explicitly presented as the chain maps 
\[
\Hom^\bullet (A^{sim} , \End_{\geq 0} (V)) \to \Hom^\bullet (A^{sim} , \End_\bullet  (V ))
\] and  
\[
\Hom^\bullet (A^{sim} , \End_{\geq 0} (V)) \to \Hom^\bullet (A^{sim} , \End_0  (V ))
\]
induced by the natural embedding $\End_{\geq 0}(V )=\sigma^{\geq 0}\End_\bullet(V) \to \End_\bullet(V )$ and the projection $\End_{\geq 0}(V ) \to\End_0(V)$.

\section{Shifted symplectic and Poisson structures}
This section is the main body of the paper. We start by reviewing the notions of shifted symplectic structure, Lagrangian structure and shifted Poisson structure following \cite{PTVV} and \cite{Sp16}. Then we will prove that $\dCx$ carries a $1$-shifted Poisson structure. This result leads to the first main result of the paper (Theorem \ref{Mainthm1}) that the moduli space of complexes of vector bundles (up to chain isomorphisms) on a $d$-dimensional Calabi Yau has a $(1-d)$-shifted Poisson structure. In the second part, we analyze the $d=1$ case in detail and prove the second main result (Theorem \ref{Lagpoint}) that says a stacky point is Lagrangian.

\subsection{Shifted Poisson structures from Lagrangian structures}
Let us recall some definitions and results in the theory of shifted symplectic and Poisson structures following \cite{PTVV} and \cite{Sp16}. Let $X$ be a derived Artin stack over $k$. We denote the tangent (resp. cotangent) complex  of $X$ by $\TT_X$ (resp. $\LL_X$). We can form the de Rham algebra
$DR(X):=\Sym^*_{\cO_X}(\LL_X[1])$. This is a weighted sheaf whose weight $p$ component is $DR(X)(p):=\Sym^p_{\cO_X}(\LL_X[1])=(\wedge^p \LL_X)[p]$.

The \emph{space of p-forms of degree n} on $X$ is a simplicial set
\[
\cA^p(X,n):= |\Hom_{L_{Qcoh(X)}}(\cO_X,\wedge^p\LL_X[n])|
\]
where $L_{Qcoh(X)}$ is the $\infty$-categorical version of the quasi-coherent category and $|\cdot |$ is the Dold-Kan denormalization.

The \emph{weighted negative cyclic chain complex}, denoted by $NC^w$, is defined to be $\bigoplus_p NC^w(X)(p)$  whose weight $p$ component is
\[
NC^w(X)(p):=(\prod_{i\geq 0} \left(\wedge^{p+i}\LL_X\right) [p-i],d_{\LL_X}+d_{DR}).
 \] 
 
The \emph{space of closed p-forms of degree n} on $X$ is 
\[
\cA^{p,cl}(X,n):=| \Hom_{L_{Qcoh(X)}}(\cO_X, NC^w(X)[n-p](p))|.
\]
The natural projection
\begin{equation}\label{formpr}
NC^w(X)[n-p](p)\to \wedge^p \LL_X[n]
\end{equation} 
defines a map $\cA^{p,cl}(X,n)\to \cA^p(X,n)$, send a closed $p$-form to the ``underlying $p$-form''. A $2$-form $\omega:\cO_X\to \wedge^2\LL_X[n]$ is \emph{non-degenerate} if the induced map $\TT_X\to \LL_X[n]$ is a quasi-isomorphism. An \emph{$n$-shifted symplectic form} on $X$ is a closed $2$-form whose underlying form is non-degenerate.

\begin{definition} (Definition 2.7 \cite{PTVV})\label{isotrop}
Let $Y$ be a derived Artin stack with an $n$-shifted symplectic form $\omega$ and let $f:X\to Y$ be a morphism. An \emph{isotropic structure} on $f$ is a path (homotopy) $h: 0\sim f^*\omega$ in the space $\cA^{2,cl}(X,n)$.
\end{definition}
The \emph{relative} tangent complex of $f$, denoted by $\TT_f$ is defined by the exact triangle 
\[
\xymatrix{
\TT_f\ar[r] &\TT_X\ar[r] & f^*\TT_Y\ar[r] & \TT_f[1]
}
\]

\begin{lemma}\label{Thetah}
An isotropic structure on $f$ defines a map $\Theta_h: \TT_f\to \LL_X[n-1]$.
\end{lemma}
\begin{proof}
The $n$-shifted symplectic structure $\omega$ defines a morphism
\[
\omega: \TT_Y\to \LL_Y[n].
\]
By the isotropic condition, 
\[
\xymatrix{
 \TT_X\ar[r] &f^*\TT_Y\ar[r]^{f^*\omega} &f^*\LL_Y[n]\ar[r] & \LL_X[n]
 }
\] factors through the map 
$\TT_f[1]\to f^*\LL_Y[n]\to \LL_X[n]$, 
whose shift defines $\Theta_h$.
\end{proof}

\begin{definition}(Definition 2.8 \cite{PTVV})\label{Lag}
We say an isotropic structure $h$ is \emph{Lagrangian} if $\Theta_h$ is a quasi-isomorphism. And we denote the simplicial set of all Lagrangian structures on $f$ by $Lagr(f,\omega)$.
\end{definition}

The following definition of $n$-shifted Poisson structures is different from the one in Section 3.1 of \cite{CPTVV}. The equivalence of the two definitions was announced by Costello and Rozenblyum. In \cite{MelSaf16} Melani and Safronov prove that
the definition below gives rise to an $n$-shifted Poisson structure in the sense of \cite{CPTVV}.

\begin{definition}\label{shiftedPoi} (Definition 2.2 \cite{Sp16})
An $n$-shifted Poisson structure on $X$ is a tuple $(Y,\omega,f,h)$, where $Y$ is a formal derived stack with an $(n+1)$-shifted symplectic structure $\omega$, $f:X\to Y$ is a morphism such that $X_{red}\to Y_{red}$ is an isomorphism, and $h$ is a Lagrangian structure on $f$.
\end{definition}

The above definition involves the theory of formal derived stacks.  We refer to Section 2.1 of \cite{CPTVV} for the details of the theory of formal derived stacks, including the definition of the reduced stack $X_{red}$ and the definition of the formal completion $\wh{Y}_X$. 

The following lemma will be our main tool of constructing shifted Poisson structures.

\begin{lemma}(Lemma 2.3 \cite{Sp16}) \label{LagPoi}
Let $X$ be a derived stack and $Y$ a derived stack with an $(n+1)$-shifted symplectic structure. Let $f:X\to Y$ be a map with a Lagrangian structure. Then $X$ has an $n$-shifted Poisson structure given by $\wh{f}: X\to\widehat{Y}_X$. We say $X$ has an $n$-shifted Poisson structure over $Y$.
\end{lemma}

To produce new Lagrangian structures we will use the following result.

\begin{lemma}\label{Lagrangian-lem} Let $S_1$ and $S_2$ be two derived Artin stacks, both equipped with an $n$-shifted symplectic
structure, and let $f_i:X\to S_i$, $i=1,2$, be morphisms, such that $(f_1,f_2):X\to S_1\times S_2$ is equipped with a
Lagrangian structure. Assume also that $g:L_1\to S_1$ is equipped with a Lagrangian structure. Let us consider
the homotopy fiber product diagram
\[
\xymatrix{
F\ar[r]^{i}\ar[d]^{\pi} & X\ar[d]^{f_1}\\
L_1\ar[r]^{g} & S_1
}
\]
Then the composition $F\to X\to S_2$ has a Lagrangian structure.
\end{lemma}

\begin{proof}
Let $\omega_1$ and $\omega_2$ be the $n$-shifted symplectic forms on $S_1$ and $S_2$.
Since $(f_1,f_2)$ is isotropic, we have a homotopy
\[
f_1^*\omega_1\sim-f_2^*\omega_2.
\]
So we have
\[
i^* f_2^*\omega_2\sim-i^*f_1^*\omega_1=-\pi^*g^*\omega_1\sim 0,
\]
where the last equality follows from the assumption that $g$ is Lagrangian. This shows that $f_2\circ i$ is isotropic. 

The definition of $F$ as a homotopy fiber product gives an identification of the relative tangent complexes
\begin{equation}\label{T-ix-eq} 
\TT_{i}\simeq \pi^* \TT_{g}.
\end{equation}
Hence, the Lagrangian structure on $g$ induces an isomorphism
\[ \TT_{i}\simeq \pi^*\LL_{L_1}[n-1].\]
Dually we get an exact triangle
\begin{equation}\label{ext1}
\pi^*\TT_{L_1}[-n]\to i^*\LL_{X}\to \LL_{F}\to \pi^*\TT_{L_1}[1-n].
\end{equation}
On the other hand, the Lagrangian structure on $(f_1,f_2): X\to S_1\times S_2$ gives an exact triangle
\[ \LL_{X}[n-1]\to\TT_{X}\to f_1^*\TT_{S_1}\oplus f_2^*\TT_{S_2}\to \LL_{X}[n].\]
 From this we get an exact triangle
\[ f_1^*\TT_{S_1}\to \LL_{X}[n]\to \TT_{f_2}[1]\to f_1^*\TT_{S_1}[1].\]
Now applying the octahedron axiom to the composition
\[
\TT_{F}\to i^*\TT_{X}\to i^*f_2^*\TT_{S_2}
\]
we get an exact triangle
\[i^*\TT_{f_2}[-1]\to \TT_{i}\to \TT_{f_2\circ i}\to i^*\TT_{f_2}.\]
Using the identification \eqref{T-ix-eq} of $\TT_{i}$, we can identify $\TT_{f_2\circ i}$ with the cone of
the composition $i^*\TT_{f_2}[-1]\to i^*f_1^*\TT_{S_1}[-1]\to \pi^*\TT_{g}$.
Finally, using the octahedron axiom for the latter composition we get an exact triangle
\[ \pi^*\TT_{L_1}[-1]\to i^*\LL_{X}[n-1]\to \TT_{f_2\circ i}\to \pi^*\TT_{L_1}.\]
Comparing it with \eqref{ext1} we deduce the quasi-isomorphism $\TT_{f_2\circ i}\simeq \LL_{F}[n-1]$.
\end{proof}

Let $(Y,\omega,f,h)$ be an $n$-shifted Poisson structure on $X$. Choose a quasi-inverse $\Theta_h^{-1}$ and consider
the composition
\[
\xymatrix{
\Pi_h : \LL_X[n]\ar[r]^{\Theta_h^{-1}} & \TT_f\ar[r] &\TT_X}.
\]
In the case $n=0$, taking the $0$-th cohomology, we get a morphism
\[
H^0(\Pi_h): H^0(\LL_X)\to H^0(\TT_X).
\]
We call this map  the \emph{classical shadow} of the $0$-shifted Poisson structure $(Y,\omega,f,h)$.

\begin{remark}
Definition \ref{shiftedPoi} has a strong motivation from the classical Poisson geometry. If $(X,\Pi)$ is a classical Poisson manifold, then $\Pi: T^*X\to TX$ is a Lie algebroid. Suppose it is integrable. Then the associated symplectic groupoid is a model for the quotient stack of $X$ by the singular foliation given by $\Pi$. Calaque, Pantev, T\"oen, Vaqui\'e and Vezzosi show that the quotient is a formal derived stack with $1$-shifted symplectic structure \cite{CPTVV2}, regardless whether the Lie algebroid is integrable.
\end{remark}

\subsection{Poisson structure on $\dCx$}\label{sec:PoiCplx}
First, we recall the construction of shifted symplectic structures on $\RR\Perf$ and $\dVt^\gr$.

Let us set 
\[
\fg:=\TT_{\RR\Perf}[-1], \ \ \fg^0:=\TT_{\dVt^\gr}[-1].
\]
Note that we have
$\fg_E\simeq \RsHom(E,E)$ (see Theorem \ref{MthmTV}), or globally
$$\fg\simeq \RsHom(\cE,\cE)\simeq\cE\otimes \cE^\vee,$$
where $\cE\in L_{parf}(\RR\Perf)$ is the universal perfect complex.

Similarly, we have
$\fg^0=\bigoplus_i \fg^0_i$
with $\fg^0_i\simeq\RsHom(\cE^i,\cE^i)$,
where $(\cE^i)_{i\in\Z}$ is the universal object in $L_{parf}(\dVt^\gr)$.

The composition of the trace map and the multiplication map 
\begin{align}\label{kappa}
\xymatrix{
\fg\otimes \fg\ar[r]^m & \fg\ar[r]^{\tr} & \cO_{\RR\Perf}}
\end{align} 
defines a $2$-form of degree $2$
\[
{\bigwedge}^2\TT_{\RR\Perf}\simeq \Sym^2\fg[2]\to \cO_{\RR\Perf}[2].
\]
This 2-form is clearly non-degenerate.
In Section 2.3 \cite{PTVV}, Pantev, T\"oen, Vaqui\'e and Vezzosi show the above bilinear form is the underlying 2-form of the closed 2-form $Ch(\cE)_2$, the 2nd Chern character of the universal perfect complex $\cE$. 

\begin{theorem} (Theorem 2.12 \cite{PTVV})\label{symperf}
The closed 2-form $Ch(\cE)_2$ defines a 2-shifted symplectic structure on $\RR\Perf$.
\end{theorem}
 
We briefly recall the proof of \cite{PTVV} for our convenience. The Atiyah class of $\cE\in L_{parf}(\RR\Perf)$
\[
a_\cE: \cE\to \cE\otimes_{\cO_{\RR\Perf}}\cE\otimes_{\cO_{\RR\Perf}}\cE^\vee
\]
 is the adjoint of the multiplication map
\[
\cE\otimes \fg \to \cE.
\] 
In \cite{TVe15}, T\"oen and Vezzosi gave a categorical construction of Chern character, which is a morphism of derived stacks
\begin{equation}\label{Ch}
Ch: \RR\Perf\to |NC|,
 \end{equation}  
 and its weight two piece is
 \[
 Ch(\cE)_2\in H^0(NC^w(\RR\Perf)(2))\simeq \pi_0(\cA^{2,cl}(\RR\Perf,2)).
 \]
 Here $|NC|$ is the simplicial set associated to $NC$ by applying the Dold-Kan denormalization.
 We refer to \cite{TVe15} for the details of the construction of $Ch$.
Denote the closed 2-form $Ch_2$ by $\kappa$. 

  Composing the map \ref{Ch} with the projection \ref{formpr}, the underlying 2-form of $\kappa$ is
\[
 \frac{\tr(a^2_\cE)}{2} \in H^2(\RR\Perf, \wedge^2\LL_{\RR\Perf}).
\]
Observe that 
\[
a^2_\cE:\cE\otimes \fg\otimes \fg\to \cE\otimes \fg\to \cE
\] is adjoint to the multiplication map $\fg\otimes \fg\to \fg$. Therefore the underlying 2-form of $\kappa$ is equal to $\frac{1}{2}\tr\circ m$.
As an abuse of notations, we will use $\kappa$ to denote both the closed 2-form and its underlying bilinear form \eqref{kappa} (multiplying by $1/2$).

Let $(\cE^i)_{i\in\Z}$ be the universal object in $L_{parf}(\dVt^\gr)$. We define bilinear forms
\[\xymatrix{
\alpha_i: \fg^0_i\otimes \fg^0_i\ar[r]^m &\fg^0_i\ar[r]^{\tr}&\cO_{\dVt^\gr}}
\] for $i\in\ZZ$.

 By the identification $\Vect\simeq \bigsqcup_{n} BGL_n$ and Section 1.2 \cite{PTVV}, 
 \[
 \alpha:=\frac{1}{2}\sum_i (-1)^{i+1}\alpha_i \in \pi_0(\cA^{2,cl}(\dVt^\gr,2))
  \]  defines a 2-shifted symplectic structure on $\dVt^\gr$. 
  Indeed, this follows from Theorem \ref{symperf} by interpreting $\alpha$ as $-s^*\kappa$, where $s$ is the map from $\dVt^\gr$ to $\RR\Perf$ defined in Section \ref{sec_Rperf}.
 
Combing these two constructions, 
\[
\omega:=(\kappa,\alpha)\in\pi_0(\cA^{2,cl}(\RR\Perf\times \dVt^\gr,2))
\]
defines a 2-shifted symplectic structure on $\RR\Perf\times \dVt^\gr$. 

Recall that we have morphisms $p,q$ and $s$ between $D^-$-stacks:
 \begin{align}\label{dig1}
 \xymatrix{
 &\dCx\ar[ld]_q\ar[rd]^p&\\
 \RR\Perf && \dVt^\gr\ar[ll]_s
 }
 \end{align}
 An object $E$ in $\dCx$ is a bounded complex of vector bundles $\ldots\to E(i)\to  E(i+1)\to\ldots$, where $E(i)$ is the $i$-th weight component.  We have $q(E)=\{\ldots\to E^i\to  E^{i+1}\to\ldots\}$ with $E^i=E(i)$ treated as a perfect complex, and $p(E)=\bigoplus_i E(i)[-i]$. The map $s$ sends a graded vector bundle $\bigoplus_i E(i)[-i]$ to $\oplus_iE^i[-i]$. 
Note that the above triangle is \emph{not} commutative. Because $E(i)$ has trivial internal differential, we will identify $\{\ldots\to E^i\to  E^{i+1}\to\ldots\}\in \RR\Perf$ and $\{\ldots\to E(i)\to  E(i+1)\to\ldots\}\in \dCx$ in the object level. However, we must keep in mind that the morphism spaces are different.
 
 By abuse of notation we denote
the pull-back of the shifted tangent complex $q^*\TT_{\RR\Perf}[-1]$ (resp., $p^*\TT_{\dVt^\gr}[-1]$)
in $L_{parf}(\dCx)$ still by $\fg$ (resp., $\fg^0$).  

Let us denote the shifted tangent complex $\TT_{\dCx}[-1]$ by $\fp^+$. By the calculation of tangent complex in the end of Section \ref{sec_mod_cplx}
$$\fp^+_E\simeq \sigma^{\geq 0}\fg_E$$
where $\sigma^{\geq 0}$ is the stupid truncation of complexes,
and we have natural morphisms 
\[
\fp^+\to \fg, ~~~ \fp^+\to \fg^0.
\]

Let 
\[f:=(q,p): \dCx\to \RR\Perf\times \dVt^\gr
\] with $q,p$
being the morphisms of $D^-$-stacks in \ref{dig1}.
 \begin{remark}
Although $\RR\Perf$ is strictly speaking not a derived Artin stack (since it is only locally geometric), symplectic structure can be anyway defined. We refer to \cite[Section 2.3]{PTVV} for the explanation of this technical point.
\end{remark}

\begin{theorem}\label{Lagstr}
There exists a Lagrangian structure on $f$.
\end{theorem}
\begin{proof}
We split the proof to two steps. First, we show that there exists a homotopy $h:0\sim f^*\omega$. This is a property of the closed 2-form. In the second step, we check that the induced map $\Theta_h$ is a quasi-isomorphism. This is a property of the underlying 2-form.

To prove the existence of a homotopy $h:0\sim f^*\omega$, it suffices to show that $f^*\omega$ represents the zero class in $\pi_0(\cA^{2,cl}(\dCx,2))$.
Consider an alternative morphism of $D^-$-stacks
\[
q^\p:\dCx\to \RR\Perf
\] by sending $\cE$ to $\bigoplus_i \cE^i[-i]$. 
 
 \begin{align*}
 \xymatrix{
 &&\dCx\ar[ld]_{q^\p}\ar[rd]^{p}&\\
 |NC|&\RR\Perf\ar[l]_{Ch} && \dVt^\gr\ar[ll]_{s}
 }
 \end{align*}
 
The key observation is  
\begin{enumerate}
\item[(1)]
$q^\p=s\circ p$,
\item[(2)] $Ch(q(\cE)))=Ch(q^\p(\cE))$ in $\pi_0(Map(\dCx,|NC|))$.
\end{enumerate}
Observation (1) is clear. Observation (2) follows from the functorial property of Chern character with respect to the exact triangle. We refer to \cite{TVe15} for the proof of this statement.
Then
\begin{align*}
f^*\omega&=q^*\kappa+p^*\alpha=q^*\kappa-p^*s^*\kappa\\
&=Ch_2(q(\cE))-Ch_2(q^\p(\cE))\\
&=0 \in\pi_0(\cA^{2,cl}(\dCx,2)).
\end{align*}
This finishes the proof of the first part.

We pick any path $h: 0\sim f^*\omega$. Recall that $q^*\TT_{\RR\Perf}[-1], \TT_{\dCx}[-1]$ and $p^*\TT_{\dVt^\gr}[-1]$ are denoted by $\fg, \fp^+$ and $\fg^0$ respectively. So $f^*\TT_{\RR\Perf\times \dVt^{\gr}}[-1]$ is equal to $\fg\oplus \fg^0$.
By Lemma \ref{Thetah}, there is a commutative diagram
\[
\xymatrix{
\fp^+\ar[r]^{\Delta}& \fg\oplus\fg^0\ar[r]\ar[d]^{\omega[-1]} &(\fg\oplus\fg^0)/\fp^+\ar[d]^{\Theta_h[-1]}\\
& (\fg\oplus\fg^0)^\vee\ar[r]^{\Delta^\vee} & (\fp^+)^\vee
}
\]
where $(\fg\oplus\fg^0)/\fp^+$ is the mapping cone of $\Delta$.
Here $\omega[-1]$ is the (shift) of underlying 2-form, defined by the bilinear form $\kappa\oplus\alpha$. The degree zero part of $\Delta$ is simply the diagonal embedding of $\fg^0\subset \fp^+$ into $\fg^0\oplus\fg^0$. Therefore, $\fp^+$ is maximal isotropic with respect to $\kappa\oplus\alpha$. In other word, $\TT_{\dCx}$ is Lagrangian in $f^*\TT_{\RR\Perf\times \dVt^{\gr}}$. It follows that $\Theta_h$ is a quasi-isomorphism.

\end{proof}
\begin{remark}
By the exactness of the first row, there is a map from $(\fg\oplus\fg^0)/\fp^+$ to $\fp^+[1]$. Compose it with a quasi-inverse of $\Theta_h[-1]$. We get the degree one bi-vector field $\Pi_h: (\fp^+)^\vee\to \fp^+[1]$. 
\end{remark}

\begin{remark}
It would be interesting to figure out whether $\pi_1(\cA^{2,cl}(\dCx,2),0)$ vanishes or not. If it does not vanish then non-homotopic paths might give different isotropic structures.
 \end{remark}
 
From Theorem \ref{Lagstr} and Lemma \ref{LagPoi}, we get the following important corollary.
\begin{theorem} \label{1-Poi}
The derived stack $\dCx$ has a $1$-shifted Poisson structure.
\end{theorem}
 
Using the techniques developed in \cite{PTVV}, we can pull back the Poisson structure constructed in Theorem \ref{1-Poi} to the mapping spaces with Calabi-Yau source. 

For $X,Y\in D^-St(k)$, denote $\Map(X,Y)$ for the derived mapping stacks between $X$ and $Y$. For its precise definition,  we refer to Section 2.2.6.3 of \cite{HAGII}. There is a natural evaluation map
\[
ev: X\times \Map(X,Y)\to Y.
\]
A very useful technique to construct new shifted symplectic structure out of a given one is by pulling back the shifted symplectic form via $ev$ and integrating along the fiber. For this to work, one needs some additional structures on $X$ called $\cO$-compactness and $d$-orientation. 
The definition of $\cO$-compactness and $d$-orientation can be found in Section 2.1 of \cite{PTVV}. We just remark that a smooth projective Calabi-Yau $d$-fold $X$ is $\cO$-compact with $d$-orientation. The choice of a $d$-orientation $[X]$ is equivalent with a choice of the isomorphism $\omega_X\cong \cO_X$. 

We recall one of the main results in \cite{PTVV}.
\begin{theorem}(Theorem 2.5 \cite{PTVV})\label{int_omega}
Let $Y$ be a derived Artin stack equipped with an $n$-shifted symplectic form $\omega$. Let $X$ be an $\cO$-compact derived stack equipped with an $d$-orientation $[X]$. Assume that the derived mapping stack $\Map(X,Y)$ is itself a derived Artin stack locally of finite presentation over $k$. Then $\Map(X,Y)$ carries a canonical $(n-d)$-shifted symplectic structure.
\end{theorem}

The construction of the symplectic form on $\Map(X,Y)$ can be summarized as follows. The $d$-orientation $[X]$ allows one to define a map 
\[
\int_{[X]}: NC^\omega(X\times \Map(X,Y))\to NC^\omega(\Map(X,Y))[-d].
\]
Then the symplectic structure on $\Map(X,Y)$, denoted by $\int_{[X]}\omega$, is defined to be the composition 
\[
\xymatrix{ k[2-n](2)\ar[r]^\omega & NC^\omega(Y)\ar[r]^{ev^*} &NC^\omega(X\times\Map(X,Y))\ar[r]^{\int_{[X]}} & NC^\omega(\Map(X,Y))[-d].
}
\]

Using a similar idea, we may also pull back Lagrangian structure via the evaluation map. The following theorem is due to Calaque.
\begin{theorem} (\cite[Theorem 2.9 ]{Sp16}) \label{LagMap}
Let $X$, $Y$, $Z$ be derived Artin stacks and $f : Y\to Z$ a map. Assume $X$ is $\cO$-compact with $d$-orientation $[X]$. Assume the stacks $\Map(X,Y)$ and $\Map(X,Z)$ are derived Artin stacks locally of finite presentation over $k$. Then for any $g:X\to Y$ we have a map: 
\[
  \int_{[X]} ev^*(g) : Lagr(f, \omega) \to Lagr(f \circ g,   \int_{[X]} ev^*(\omega)),
  \]
that is, from Lagrangian structures on $f$ to Lagrangian structures on $f\circ g$.
\end{theorem}
Another important result of \cite{PTVV} says that the intersection of Lagrangians in a shifted symplectic derived stack is again shifted symplectic.

\begin{theorem}(\cite[Theorem 2.9]{PTVV}) \label{Lagintersection}
Let 
\[
\xymatrix{
& Y\ar[d]^g\\
X\ar[r]_f & F,
}\] be a diagram of derived Artin stacks, $\omega$ be an $n$-shifted symplectic structure on $F$, and $f$ and $g$ be equipped with Lagrangian structures. Then the derived Artin stack $X\times^h_F Y$ is equipped with a canonical $(n-1)$-shifted symplectic structure.
\end{theorem}

In this paper, we consider the target spaces being $\dCx,\RR\Perf$ and $\Vect$ (or $\dVt^\gr$). For simplification, we write $Z(X)$ for $\Map(X,Z)$. Thus, $\dCx(X)$, $\RR\Perf(X)$ and $\RR\dVt(X)$ (or $\RR\dVt^\gr(X)$) denote for the derived moduli stacks of complexes, perfect complexes and vector bundles (or graded vector bundles) on $X$, respectively. 
Note that $\RR\dVt(X )$ (or $\RR\dVt^\gr(X)$) is derived only when $\dim(X)>1$.

Combining Theorems \ref{Lagstr} and
\ref{LagMap} we obtain the following result about moduli space of complexes on Calabi-Yau manifolds.

\begin{theorem}\label{Mainthm1}
Let $X$ be a smooth projective CY $d$-fold. Then the natural map $\dCx(X)\to \RR\Perf(X)\times \RR\dVt^\gr(X)$ has a 
Lagrangian structure. In particular, $\dCx(X)$ has a $(1-d)$-shifted Poisson structure. 
\end{theorem}

\subsection{Case $d=1$}
Let $k$ be a field of characteristic zero and $C$ be a smooth elliptic curve over $k$. The derived moduli space $\RR\Perf(C)$ classifies perfect complexes on $C$, i.e. isomorphism classes of objects in $\D^b(\coh~C)$. And $\dVt(C)$ classifies vector bundles on $C$. In this case, there is a fairly clear understanding about the $0$-shifted Poisson structure on $\dCx(C)$. In fact, many classical Poisson structures have appeared in algebraic geometry and integrable system are examples of this type. We will study them in detail in the next section.

The second main result of this section is the following.
\begin{theorem} \label{Lagpoint}
Let $\cF$ be a vector bundle on $C$ and $\cE^\bullet$ be a perfect complex on $C$. Denote $x_{\cE^\bullet}$ and $y_\cF$ for the stacky point in $\RR\Perf(C)$ and $\dVt(C)$, representing $\cE^\bullet$ and $\cF$ respectively
(see Def.\ \ref{stacky-pt-def}). Then the natural monomorphisms 
\[
j_x: x_{\cE^\bullet}\to \RR\Perf(C)~~~\text{and}~~~ j_y: y_{\cF}\to \dVt(C)
\] are Lagrangian with respect to the $1$-shifted symplectic structures on the target.
\end{theorem}
\begin{proof}
We consider the second map first. The pull back of the tangent complex of $\dVt(C)$ to $y_\cF$, denoted by $\TT_{\cF}$ is quasi-isomorphic to
$\Gamma(C^{cos},\sEnd(\cF))[1]$, where $C^{cos}$ is a cosimplicial resolution of $C$. Here we choose the Cech model. Then $\Gamma(C^{cos},\sEnd(\cF))$ is the Cech complex that computes $\Ext^*(\cF,\cF)$.

The 2-shifted symplectic form $\alpha\in\cA^{2,cl}(\Vect,2)$ corresponds to a morphism of graded complexes
\[
\alpha: k(2)\to NC^\omega(\Vect).
\]
The composition
\[
\xymatrix{
\alpha_C: k(2)\ar[r]^\alpha & NC^\omega(\Vect)\ar[r]^{ev^*} &NC^\omega(C\times\dVt(C))\ar[r]^{\int_{[C]}} & NC^\omega(\dVt(C))[-1]
}
\] defines a $1$-shifted symplectic structure.

Let $G$ be a affine group scheme over $k$. We denote its Lie algebra by $\fa$.
We have
\[
NC^w(BG)[n-p](p)=(\prod_{i\geq 0}\Lambda^{p+i}\fa^\vee[n-p-2i],d+d_{dR}).
\]
When $n=1$ and $p=2$,  $\pi_0(\cA^{2,cl}(BG,1))$ vanishes for degree reason.
We take $G=Aut(\cF)$. Then $y_\cF$ is isomorphic with $BG$ as stacks, by \ref{vect=BG}. The vanishing of $\pi_0$ implies that there exists a homotopy $h: j_y^*\alpha_C\sim 0$, which shows that $j_y$ is isotropic. The tangent complex of $y_\cF$ is isomorphic to $\tau^{<0}\TT_\cF\simeq \Ext^0(\cF,\cF)[1]$. And the relative tangent complex $\TT_{j_y}$ is isomorphic to $\tau^{\geq 0}\TT_\cF\simeq\Ext^1(\cF,\cF)[-1]$. By Lemma \ref{Thetah}, we have a chain map
\[
\Theta_h: \TT_{j_y}\simeq \tau^{\geq 0}\TT_\cF\to (\tau^{<0}\TT_\cF)^\vee[1].
\]
The Serre duality on $C$ implies that $\Theta_h$ is a quasi-isomorphism.

Now we prove the first part. Denote the pull back of the tangent complex of $\RR\Perf$ via map $j_x$ by $\TT_{\cE^\bullet}$. It is quasi-isomorphic to the shift of Cech complex
$\Gamma(C^{cos},\sEnd^\bullet(\cE^\bullet))[1]$. Recall that the 2-shifted symplectic form $\kappa\in\cA^{2,cl}(\RR\Perf,2)$ is given by the weight 2 part of the categorical Chern character $Ch: \RR\Perf\to |NC|$. The composition
\[
\xymatrix{
\kappa_C: k(2)\ar[r]^\kappa & NC^\omega(\RR\Perf)\ar[r]^{ev^*} &NC^\omega(C\times\RR\Perf(C))\ar[r]^{\int_{[C]}} & NC^\omega(\RR\Perf(C))[-1]
}
\] defines a $1$-shifted symplectic structure.

The tangent complex of $x_{\cE^\bullet}$ is quasi-isomorphic to $\tau^{<0}\TT_{\cE^\bullet}$. Because
\[
NC^w(x_{\cE^\bullet})[n-p](p)=(\prod_{i\geq 0}\Lambda^{p+i}(\tau^{<0}\TT_{\cE^\bullet})^\vee[n-i],d+d_{dR}),
\]
for $n=1$ and $p=2$, $\pi_0(\cA^{2,cl}(x_{\cE^\bullet},1))$ vanishes for degree reason. So there exists a homotopy $l: j_x^*\kappa_C\sim 0$. The induced chain map 
\[
\Theta_l: \TT_{j_x}\simeq \tau^{\geq 0}\TT_{\cE^\bullet}\to (\tau^{<0}\TT_{\cE^\bullet})^\vee[1]
\] is a quasi-isomorphism again by Serre duality.
\end{proof}

\begin{remark}
We may compare the Lagrangian structure of Theorem \ref{Lagpoint} and that of Theorem \ref{Lagstr}. They are of very different nature. The Lagrangian structure in Theorem \ref{Lagstr} comes from an analogue of the parabolic structure on groups, while the Lagrangian structure in Theorem \ref{Lagpoint} comes from the geometry of elliptic curves.
\end{remark}

Combining Theorems \ref{Mainthm1} and \ref{Lagpoint} with Lemma \ref{Lagrangian-lem}
we deduce the following result giving more examples of $0$-shifted Poisson structures.
For a subset $S\subset \ZZ$ let us denote by $\dVt^S$ the derived stack of $S$-graded vector bundles, so that
$\dVt=\dVt^S\times \dVt^{\overline{S}}$, where $\overline{S}\subset \ZZ$ is the complement of $S$.
We denote by $p^S:\dCx\to \dVt^S$ the composition of $p$ with the natural projection.

\begin{corollary}\label{Casi}
For a subset $S\subset\ZZ$, let
$x_{\cE^\bullet}$, $y_{\cF^S}$ and $y_{\cF^{\overline{S}}}$ denote the stacky points in $\RR\Perf(C)$, $\dVt^S(C)$ and 
$\dVt^{\overline{S}}(C)$, corresponding to some objects $\cE^\bullet$, $\cF^S$ and $\cF^{\overline{S}}$, respectively. Denote the homotopy fiber products of the diagrams
\[
\xymatrix{
 & \dCx(C)\ar[d]^{(q,p^S)} &&  & \dCx(C)\ar[d]^{p^{\overline{S}}}\\
x_{\cE^\bullet}\times y_{\cF^S} \ar[r]^{} & \RR\Perf(C)\times \dVt^S(C) && y_{\cF^{\overline{S}}}\ar[r]^{} & \dVt^{\overline{S}}(C)
}
\] by $F_{x,y^S}$ and $F_{y^{\overline{S}}}$ respectively. Then the compositions
\[
F_{x,y^S}\to \dCx(C)\to \dVt^{\overline{S}}(C),~~\text{and}~~ F_{y^{\overline{S}}}\to\dCx(C)\to \RR\Perf(C)\times \dVt^S(C)
\] are Lagrangian. Hence $F_{x,y^S}$ and $F_{y^{\overline{S}}}$ have $0$-shifted Poisson structures.
\end{corollary}

\begin{corollary}\label{leaves}
The homotopy fiber product $F_{x,y}$ of the diagram
\[
\xymatrix{
 && \dCx(C)\ar[d]^{f}\\
(x_{\cE^\bullet},y_{\cF^\bullet})\ar[rr]_{(j_x,j_y)} && \RR\Perf(C)\times\dVt^\gr(C),
}
\] 
has a $0$-shifted symplectic structure. 
\end{corollary}

\begin{proof}
This follows from Theorem \ref{Lagpoint} and Theorem \ref{Lagintersection}.
\end{proof}

As an analogue of the terminology in classical Poisson geometry, we call $F_{x,y}$ the \emph{symplectic leaves} of the $0$-shifted Poisson structure on $\dCx(C)$. Note that in algebraic setting, the existence of algebraic symplectic leaves is in general not known even for smooth Poisson varieties. Corollary \ref{leaves} provides a potential way to prove such existence.

\begin{remark}
To give a link to the classical Poisson geometry, we may assume all the stacks in the previous two corollaries are smooth schemes. Then the Corollary \ref{Casi} claims that $p$ and $q$ are Casimir maps. If we further assume that the intersection of the fibers of $p$ and $q$ are smooth schemes, then Corollary
\ref{leaves} claims that the intersections of the fibers are symplectic leaves. We will see in the next section that Corollary \ref{leaves} is a powerful tool to compute the symplectic leaves of many classical Poisson structures arising from elliptic curves. 
\end{remark}


\section{Example one: moduli spaces of torsion free sheaves on $\PP^2$ and their deformations}
As the first application, we show that by considering 
$0$-shifted Poisson structure constructed in the previous section 
in the case of certain $3$-term complexes on elliptic curves one recovers Poisson structures on the 
moduli spaces of semistable sheaves on $\PP^2$ and their deformations, constructed
by Nevins and Stafford \cite{NS06}. In the case of Hilbert schemes, we use Corollary \ref{leaves} to study their symplectic leaves. In this section, $k$ is taken to be the field of complex number $\CC$.

\subsection{$0$-shifted Poisson structures on moduli spaces of torsion free sheaves on $\PP^2$ and their deformation}
Let us recall the monad construction for torsion free sheaves on complex projective plane following \cite[Ch.\ 2]{Nak99}.

Let us denote $c_1(\cO_{\PP^2}(1))$ by $H$. Recall that a 
torsion free sheaf $E$ on $\PP^2$ is called \emph{normalized} if $-\rk(E)<c_1(E)\cdot H\leq 0$. Note that 
for arbitrary $E$ there exists a unique $n\in\Z$ such that $E(n)$ is normalized. 
The following result is well known (see \cite[Lemma 2.4]{Nak99}).

\begin{lemma}\label{normalized-coh-lem}
Let $E$ be a normalized semistable torsion free sheaf on $\PP^2$, then
\begin{align*}
\begin{cases}
H^q(\PP^2,E(-p))=0 & \text{for} ~~ p=1,2, q=0,2\\
H^q(\PP^2,E(-1)\otimes \cQ^\vee)=0 & \text{for} ~~  q=0,2.
\end{cases}
\end{align*}
\end{lemma}

The above lemma together with the Beilinson spectral sequence leads to the following result (see \cite[Section 2.1]{Nak99}).
Let $\cQ$ denote the twisted tangent bundle $T_{\PP^2}(-1)$. 

\begin{prop}\label{monaddes}
Let $E$ be a normalized semistable torsion free sheaf on $\PP^2$. 
There exists a complex of sheaves
\[
\xymatrix{
\cO_{\PP^2}(-1)\otimes H^1(\PP^2,E(-2))\ar[r]^a &\cO_{\PP^2}\otimes H^1(\PP^2,E(-1)\otimes \cQ^\vee)\ar[r]^b & \cO_{\PP^2}(1)\otimes H^1(\PP^2, E(-1)),
}
\] such that $a$ is injective, $b$ is surjective and $\ker ~b/\im ~a\cong E$.
\end{prop}

More generally, for finite dimensional vector spaces $(V_i)$, $i=-1,0,1$, we call a complex of the form
\begin{equation}\label{KP2}
\bK:~~\xymatrix{
\cO_{\PP^2}(-1)\otimes V_{-1}\ar[r]^a &\cO_{\PP^2}\otimes V_0\ar[r]^b & \cO_{\PP^2}(1)\otimes V_1
}
\end{equation} 
a \emph{Kronecker complex}. 
Two Kronecker complexes are called isomorphic if they are isomorphic as complexes of sheaves (rather than objects in the derived category). A Kronecker complex is called a \emph{monad} if $a$ is injective, $b$ is surjective and the middle cohomology sheaf  is torsion free.

Let us denote the moduli stack of semistable torsion free sheaves on $\PP^2$ of rank $r$, degree $-r<d\leq0$ and $c_2=-n$ by $\cM^{ss}(r,d,n)$.
By Proposition \ref{monaddes}, $\cM^{ss}(r,d,n)$ is an open substack of the moduli stack of Kronecker complexes on $\PP^2$. 
By Lemma \ref{normalized-coh-lem}, the dimensions of the vector spaces $V_i$ are determined by $r,d,n$ via the Riemann-Roch formula. For example, if we set $d=0$ then the Kronecker complex corresponding to $\cF\in\cM^{ss}(r,0,n)$ has the 
spaces $V_{-1},V_0,V_1$ of dimensions $n,2n+r,n$, respectively.

\begin{theorem} (Proposition 9.2 \cite{NS06}) \label{equiv-stack}
Let $C$ be a plane cubic. A Kronecker complex on $C$ is defined to be a complex of the form
\begin{equation}\label{Kcplx}
\bK_C: \xymatrix{
\cL^\vee\otimes V_{-1}\ar[r] &\cO_C\otimes V_0\ar[r] & \cL\otimes V_1}
\end{equation}
for some line bundle $\cL$ of degree 3 on $C$. Let $\bK$ be a Kronecker complex on $\PP^2$ of the form \eqref{KP2} and define $\cL=\cO(1)|_C$.
The restriction functor induces an equivalence between the moduli functor of Kronecker complexes on $\PP^2$ and the moduli functor of Kronecker complexes on $C$.
\end{theorem}
The above theorem follows from the observation that 
\[
\Hom_{\PP^2}(\cO,\cO(i))\cong\Hom_C(\cO_C,\cO_C(i)) \hs \text{for} \hs i=0,1,2.
\]
In particular, if $\bK$ is a monad then $\bK_C$ has only middle cohomology (see Corollary 9.2 \cite{NS06}).

We can apply Theorem \ref{Mainthm1} to the case when $X$ is a smooth plane cubic $C$. Given a degree 3 line bundle $\cL$ on $C$ and finite dimensional vector spaces $V_{-1},V_0,V_1$, let $y$ be the stacky point in $\dVt^\gr(C)$ corresponding to $(\cL^\vee\otimes V_{-1},\cO_C\otimes V_0,\cL\otimes V_1)$.
Then the fiber $F_y$ of 
\[
p:\dCx(C)\to \dVt^\gr(C)
\]
at $y$ is precisely the moduli stack of Kronecker complex on $C$ with the data $(\cL,V_{-1},V_0,V_1)$ fixed. Then by Corollary \ref{Casi}, we obtain
\begin{corollary}\label{PoiP2}
The moduli stack of semistable sheaves 
$\cM^{ss}(r,d,n)$ has a $0$-shifted Poisson structure.
\end{corollary}

Now let $\eta$ be a translation of $C$, and let us set $\cL^\eta:=(\eta^{-1})^*\cL$.
For fixed integers $r> 0,d,n\geq 0$ and vector spaces $V_{-1}, V_0,V_1$ of dimensions determined by $r,d,n$, let us set
\[
y_\eta:=(\cL^\vee\otimes V_{-1},\cO_C\otimes V_0,\cL^{\eta}\otimes V_1)\in \dVt^\gr(C)
\] 
 Then the fiber $F_{y_\eta}$  of the map $p:\dCx(C)\to\dVt^\gr(C)$
parameterizes complexes on $C$ of the form
\begin{equation} \label{Keta}
\bK_C^\eta: \xymatrix{
\cL^\vee\otimes V_{-1}\ar[r]^a &\cO_C\otimes V_0\ar[r]^b & \cL^\eta\otimes V_1}
\end{equation}
Therefore, the moduli stack of complexes of the form $\bK^\eta_C$ 
carries a $0$-shifted Poisson structure from Corollary \ref{Casi}.                                         

As Nevins and Stafford show, the complexes $\bK^\eta_C$ are related to the monad description of semistable sheaves on the noncommutative projective planes. 
Let $C$ be a smooth cubic in $\PP^2$ and $\cL=\cO(1)|_C$. Fix an automorphism $\eta\in Aut(C)$ given by translation under the group law.
Denote the graph of $\eta$ by $\Gamma_\eta\subset C\times C$. If $V=H^0(C,\cL)$, there is a $3$-dimensional subspace 
\[
\cR=H^0(C\times C, (\cL\bt\cL)(-\Gamma_\eta))\subset V\ot V.
\]
\begin{definition}\label{Skl3}
The \emph{3-dimensional Sklyanin algebra} is the algebra
\[
S_\eta=Skl(C,\cL,\eta)=T(V)/(\cR),
 \]  where $T(V)$ denotes for the tensor algebra of $V$ and $(\cR)$ denotes the two-sided ideal generated by $\cR$.
\end{definition}
The above definition applies even when $C$ is singular, though we will focus on the cases when $C$ is smooth. By an abuse of notation, we use the same symbol $\eta$ to represent the point in $C$ that defines the translation. As $\eta$ varies $S_\eta$ form a flat family of noncommutative algebras over the base $\cB=C$.  

For fixed $\eta$ the abelian category $\qgr (S_\eta)$ is defined to be the quotient category of the category of finitely generated (right)-$S_\eta$-modules by the subcategory of torsion modules. When $\eta$ is equal to the neutral element $o\in C$, $S_o$ is the graded polynomial algebra of three variables, and $\qgr(S_o)\simeq \coh(\PP^2)$. We will denote by $\PP^2_\eta$  the corresponding noncommutative projective plane such that
\[
\qgr(S_\eta)\simeq \coh(\PP^2_\eta).
\]
Such a $\PP^2_\eta$ is called a \emph{elliptic quantum projective plane}. 
A key fact is that $\coh(C)\simeq \qgr(S_\eta/gS_\eta)$ for a central element $g\in S_\eta$ of degree 3.
This allows us to define the restriction functor $(-)|_C$ as the right adjoint to the inclusion functor.

In \cite{NS06}, Nevins and Stafford extended the monad construction of moduli space of torsion free sheaves on $\PP^2$ to the elliptic quantum projective planes $\PP^2_\eta$. They prove the following.
\begin{theorem}(Theorem 1.6, 1.7, 1.9 \cite{NS06}) \label{NSthm1}
Let $\PP^2_\eta$ be an elliptic quantum projective plane and fix $r\geq 1$, $d\in\ZZ$ and $n\geq 0$.
\begin{enumerate}
\item[(1)] There is a (relative) projective coarse moduli space $M^{ss}_C(r,d,n)\to C$, such that the fiber over $\eta\in C$, $M^{ss}_\eta(r,d,n)$, parameterizes semistable torsion free sheaves in $\qgr(S_\eta)$ of rank $r$, degree $d$ and $c_2=-n$.
\item[(2)] There is a quasi-projective $C$-subscheme $M^{s}_C(r,d,n)\subset M^{ss}_C(r,d,n)\to C$, smooth over $C$,
such that $M^{s}_\eta(r,d,n)$ parameterizes stable sheaves.
\item[(3)] The moduli space $M^{s}_\eta(r,d,n)$ admits a natural Poisson structure.
\end{enumerate}
\end{theorem}

The key step in the proof of Theorem \ref{NSthm1} is to show that the moduli functor of torsion free sheaves on $\PP^2_\eta$ of rank $r$, degree $d$ and $c_2=-n$ is equivalent to the moduli functor of Kronecker complexes \eqref{Keta}. Then the coarse moduli space $M^{ss}_\eta(r,d,n)$ can be constructed using geometric invariant theory.
On the other hand, using Corollary \ref{Casi}, we obtain a $0$-shifted Poisson structure on the corresponding moduli stack
$\cM^{ss}_\eta(r,d,n)$. Below we will compute the classical shadow $H^0(\Pi_h)$ of this $0$-shifted Poisson structure and show that it descends to a classical Poisson structure on the coarse moduli space $M^{s}_\eta(r,d,n)$, which coincides with the Poisson structure of Nevins and Stafford. Note that the Poisson structure of Nevins and Stafford is only defined on the smooth part of the coarse moduli space. So our $0$-shifted Poisson structure carries some additional information.

\begin{theorem}\label{Thmapp1}
For every choice of smooth plane cubic $C\subset \PP^2$ and $\eta\in C$,
the moduli stack of semistable torsion free sheaves on $\PP^2_\eta$ is a $0$-shifted Poisson  substack of $\dCx(C)$. Its classical shadow coincides with the Poisson structure of Nevins and Stafford defined on the smooth locus.
\end{theorem}
\begin{proof}
\cite[Prop.\ 6.20]{NS06} shows that the moduli space of semistable torsion free sheaves on $\PP^2_\eta$ is equivalent, as Artin 1-stacks, to the moduli space of complexes of the form \eqref{Keta} with certain stability condition defined in Section 6 of \cite{NS06}.
By Corollary \ref{Casi}, the $0$-shifted Poisson structure on $\dCx(C)$ restricts to $F_{y_\eta}$. Now let us compute its classical shadow.

Let 
\[
\xymatrix{
\cE=\{ \ldots \cE^{i-1}\ar[r]^{\phi_{i-1}} &\cE^i\ar[r]^{\phi_i} & \cE^{i+1}\ar[r] &\ldots \}
}
\]
be a complex of vector bundles on $C$.
The tangent complex $\TT_{\RR\Perf(C),\cE}$ is quasi-isomorphic with $\Gamma(C^{cos},\cC^\bullet)$, where $\cC^\bullet$ is a complex of vector bundles defined by
\[
\cC^d=\bigoplus_i \sHom(\cE^i,\cE^{i+d}) ~~\text{for}~~d\in\ZZ
\]
and differential $\partial$ defined by 
\[
\partial(f^d_i)=\phi_{i+d} f^d_i-(-1)^d f^d_{i+1}\phi_i,  \text{~for~}  f^d_i \in\sHom(\cE^i,\cE^{i+d}).
\]
The tangent complex $\TT_{\dCx(C),\cE}$ is quasi-isomorphic to $\Gamma(C^{cos},\sigma^{\geq 0} \cC^\bullet)[1]$.

We denote by $g_\cE: C\to \RR\Perf$ the map that corresponds to $\cE\in \RR\Perf(C)$.
Using the general theory of mapping stacks (\cite[Sec.\ 2.1]{PTVV} for details), we can identify $\TT_{\RR\Perf(C),\cE}$  and  $\TT_{\dCx(C),\cE}$
with $\RsHom(\cO_C,g_\cE^*\fg)[1]$ and $\RsHom(\cO_C,g_\cE^*\fp^+)[1]$ where $\fg$ and $\fp^+$ are defined in Section \ref{sec:PoiCplx}. Clearly, $g_\cE^*\fg$ and $g_\cE^*\fp^+$ are represented by $\cC^\bullet$ and $\sigma^{\geq 0} \cC^\bullet$, respectively.

Recall the commutative diagram \eqref{dig1}:
\[
\xymatrix{
\fp^+\ar[r]^\Delta& \fg\oplus\fg^0\ar[r]\ar[d]^{(\kappa,\alpha)} & (\fg\oplus\fg^0)/\fp^+\ar[d]^{\Theta_h[-1]}\ar[r] & \fp^+[1]\\
& (\fg\oplus\fg^0)^\vee\ar[r] & (\fp^+)^\vee\ar[ur]^{\Pi_h}
}
\]

The morphism $\Theta_h$ is induced by the bilinear form $(\kappa,\alpha)$. Let $g_\cE^*\kappa$ and $g_\cE^*\alpha$ 
denote the pairings on $\cC^\bullet$ and $\cC^0$, obtained by pulling back $\kappa$ and $\alpha$. There is a commutative diagram of complexes of vector bundles:
\[
\xymatrix{
\cC^{\geq 0}\ar[r]^{g_\cE^*\Delta} &\cC^\bullet\oplus \cC^0\ar[r] &(\cC^\bullet\oplus \cC^0)/\cC^{\geq 0}\ar[d]_{g_\cE^*(\Theta_h[-1])}\ar[r] & \cC^{\geq 0}[1]\\
&&(\cC^{\geq 0})^\vee\ar[ur]^{g_\cE^*\Pi_h}
}
\]
Here $\cC^{\geq 0}$ and $\cC^{\leq 0}$ are defined to be $\sigma^{\geq 0} \cC^\bullet$ and  $\sigma^{\leq 0} \cC^\bullet$, and $(\cC^\bullet\oplus \cC^0)/\cC^{\geq 0}$ is the mapping cone of $g_\cE^*\Delta$. 

We define $\ad: \cC^{\leq 0}\to \cC^{\geq 0}[1]$ to be the chain map 
\[
\xymatrix{
&&\cC^0\ar[r]^\partial &\cC^1\ar[r] &\cC^2\\
\ldots\ar[r] & \cC^{-2}\ar[r] &\cC^{-1}\ar[r]^\partial\ar[u]^\partial &\cC^0\ar[u]^0
}
\]
There is an isomorphism of complexes $Cone(\ad)[-1]\cong \cC^\bullet\oplus \cC^0$, defined by 
\[
\xymatrix{
Cone(\ad)[-1]\ar[d]^a: &\ldots\ar[r] & \cC^{-1}\ar[r]^{(\partial,\partial)}\ar@{=}[d] &\cC^0\oplus \cC^0\ar[r]^{(\partial,0)}\ar[d]^{A} 
&\cC^1\ar@{=}[d]\ar[r] &\ldots\\
\cC^\bullet\oplus \cC^0:& \ldots\ar[r] & \cC^{-1}\ar[r]^{(\partial,0)} &\cC^0\oplus \cC^0\ar[r]^{(\partial,0)} &\cC^1\ar[r]&\ldots}
\]
with $A=\left(\begin{array}{cc}1&0\\ 1&-1\end{array}\right)$. Note that the diagram
\[
\xymatrix{
\cC^{\geq 0}\ar@{=}[d]\ar[r] &Cone(\ad)[-1]\ar[d]^a\\
\cC^{\geq 0}\ar[r]^{g_\cE^*\Delta} &\cC^\bullet\oplus \cC^0
}
\] commutes. It follows that there is a quasi-isomorphism $(\cC^\bullet\oplus \cC^0)/\cC^{\geq 0}\simeq \cC^{\leq 0}$.

Now we observe that $g_\cE^*\Theta_h[-1]:\cC^{\leq 0}\to(\cC^{\geq 0})^\vee$ is precisely the restriction of  $g_\cE^*\kappa$.
The degree 0 component of $(g_\cE^*\kappa)^{-1}|_{\cC^0}$, is equal to $\ft:=\sum_i t_i$, where
\[
t_i: \sHom(\cE^i,\cE^i)^\vee\to  \sHom(\cE^i,\cE^i)
\] is $(-1)^i$ times the natural duality isomorphism. So we conclude that
\[
g_\cE^*\Pi_h=\ad\circ \ft.
\]
This coincides with the map $\psi$ defined by Nevins and Stafford in \cite[Lem.\ 9.6]{NS06}.
Therefore,
\[
\HH^1(\Pi_h): \HH^1((\cC^{\geq 0})^\vee[-1])\to \HH^1(\cC^{\geq 0})
\]
matches the Poisson structure in \cite{NS06} over the smooth locus of the coarse moduli space\footnote{The construction of \cite{NS06} in fact produces a morphism $\psi:H^0(\LL_\cM)\to H^0(\TT_\cM)$ on the open substack 
$\cM\subset \cM^{ss}_\eta(r,d,n)$, which is the preimage of the smooth locus of the coarse moduli space.}.
\end{proof}

\begin{prop}\label{Poicoarse}
The classical shadow of our $0$-shifted Poisson structure descends to a classical Poisson structure along the coarse moduli functor $f: \cM^s_\eta(r,d,n)\to M^s_\eta(r,d,n)$.
\end{prop}
\begin{proof}
Let $\cF$ be a stable torsion-free sheaf in $\qgr(S_\eta)$. For simplicity, we denote $\cM^s_\eta(r,d,n)$ and  $M^s_\eta(r,d,n)$ by $\cM^s$ and $M^s$ respectively. 
It follows from \cite[Lem.\ 7.14]{NS06} that $\Ext^2_{\PP^2_\eta}(\cF,\cF)=0$. Therefore, the derived structure on $\cM^s$ is trivial. Furthermore, the morphism $f: \cM^s\to M^s$ is a $\GG_m$-gerbe, and there are natural isomorphisms
\[
f^*T_M\cong H^0(\TT_\cM),~~~f^*\Omega^1_M\cong H^0(\LL_\cM).
\]
Hence, the classical shadow of our Poisson structure is a morphism $\psi:f^*\Omega^1_M\to f^*T_M$.
But $f_*\cO_{\cM^s}\simeq \cO_{M^s}$, since $f$ is a $\GG_m$-gerbe, so $\psi$ descends to a morphism $\Omega^1_M\to T_M$.
\end{proof}

Bottacin constructed in \cite{Bo95} for every smooth projective Poisson surface $S$ a canonical Poisson structure on the smooth part of the coarse moduli space of stable torsion-free sheaves over $S$.
For $S=\PP^2$ a choice of a nonzero Poisson structure on $S$ (up to rescaling) is same as a choice of a cubic curve $C$ (possibly singular).  For smooth $C$ it can be checked that the Poisson structure of Bottacin coincides with the one obtained in Proposition \ref{Poicoarse}. We conjecture that for other surfaces there is a similar relation to the $0$-shifted Poisson structures
on the moduli stack of complexes on $C$.

\begin{conjecture}
Let $S$ be a smooth projective Poisson surface and let $C\subset S$ be a smooth anticanonical divisor such that the Poisson bivector degenerates at $C$. Then there exists an isomorphism between the moduli stack of semistable torsion-free sheaves on $S$ and a Poisson substack of the moduli stack of complexes of vector bundles on $C$, such that the $0$-shifted Poisson structure descends to the Poisson structure of Bottacin over the smooth locus of the coarse moduli space.
\end{conjecture}

\subsection{Symplectic leaves of $\cM^s_\eta(1,0,n)$}\label{ell-def-sympl-leaves-sec}

In this section we discuss the symplectic leaves of $\cM^{ss}_\eta(r,d,n)$. As before, we consider normalized sheaves, i.e., 
assume that $-r<d\leq 0$. Then $\cM^{ss}_\eta(r,d,n)$  can be identified with the moduli space of semistable Kronecker complexes
\[
\bK_C^\eta: \xymatrix{
\cL^\vee\otimes V_{-1}\ar[r]^a &\cO_C\otimes V_0\ar[r]^b & \cL^\eta\otimes V_1}\]
for an appropriate stability conditions defined in \cite[Sec.\ 6]{NS06}. It is equipped with a $0$-shifted Poisson structure via the identification of $\cM^{ss}_\eta(r,d,n)$ with an open substack of $F_y$, for $y=(\cL^\vee\otimes V_{-1},\cO_C\otimes V_0, \cL^\eta\otimes V_1)$. 
Since a semistable Kronecker complex has only middle cohomology, the map $q:F_y\to \RR\Perf(C)$ restricts to the map
\[
q: \cM^{ss}_\eta(r,d,n)\to \coh(C)
\] sending a Kronecker complex to its middle cohomology sheaf. For a sheaf $\cH\in\coh(C)$, we denote
by $F_{\cH}$ the homotopy fiber of the stacky point represented by $\cH$.

As $\cM^{s}_\eta(r,d,n)$ is smooth, we expect that the symplectic leaves of $\cM^s_\eta(r,d,n)$ should descend to symplectic leaves for the classical Poisson structure on the coarse moduli scheme $M^{s}_\eta(r,d,n)$.  This is indeed the case
assuming the coarse moduli space of a symplectic leaf is smooth.

For simplicity, we denote $\cM^s_\eta(r,d,n)$ and  $M^s_\eta(r,d,n)$ by $\cM^s$ and $M^s$ respectively. 

\begin{prop}\label{Leavescoarse}
Let $F_\cH$ be the homotopy fiber product of 
\begin{equation}\label{FH}
\xymatrix{
& \cM^s\ar[d]\\
x_\cH\ar[r] & \coh(C)}
\end{equation} where $x_\cH$ is the stacky point corresponding to $\cH\in \coh(C)$.
Assume $F_\cH$ is non-empty and the coarse moduli scheme $F_\cH^c$ of $F_\cH$ is smooth. Then the $0$-shifted symplectic structure on $F_\cH$ from Corollary \ref{leaves} descends to a classical symplectic structure on $F_\cH^c$.
\end{prop}

\begin{proof}
An important fact is that even though $\cM^s$ is un-derived (see the proof of Proposition \ref{Poicoarse}), the homotopy fiber $F_\cH$ carries a nontrivial derived structure. Given $\cF\in \cM^s$, the exact triangle of the tangent complex of the homotopy fiber product implies that $H^{-1}(\TT_{F_\cH,\cF})\cong \CC$ and gives a long exact sequence
\[
\xymatrix{
0\ar[r] & H^0(\TT_{F_\cH,\cF})\ar[r] & H^0(\TT_{\cM^s,\cF})\ar[r] & \Ext_C^1(\cH,\cH)\ar[r] & H^1(\TT_{F_\cH,\cF})\ar[r] & 0}
\]
The existence of a $0$-shifted symplectic structure on $F_\cH$ implies that $H^1(\TT_{F_\cH,\cF})\cong \CC$. In fact, we may identify $H^1(\TT_{F_\cH,\cF})$ with the tangent space of the Picard stack $Pic(C)$ at $\det(\cH)$, so that the map 
$\Ext_C^1(\cH,\cH)\to H^1(\TT_{F_\cH,\cF})$ is precisely the tangent map of $\det: \coh(C)\to Pic(C)$ at $\cH$.

The proof of the fact that the $0$-shifted symplectic structure descends to $F_\cH^c$ is similar to the proof of Proposition \ref{Poicoarse}. We simply observe that $F_{\cH}$ is a $\GG_m$-gerbe over $F_\cH^c$, and hence,
the classical shadow of the $0$-shifted symplectic structure on $F_{\cH}$ descends to $F_\cH^c$.
\end{proof}

As a concrete example, we calculate the symplectic leaves in the case $r=1, d=0$.
In this case, $\cM^{ss}_\eta(1,0,n)$ and $\cM^{s}_\eta(1,0,n)$ coincide. Also $\cM^s$ is a trivial $\GG_m$-gerbe over $M^s$. 
Following the standard notation for the Hilbert scheme of $n$ points, we denote the coarse moduli space of $\cM^s_\eta(1,0,n)$ by $(\PP^2_\eta)^{[n]}$. An object in $\cM^s$ is a right ideal of $S_\eta$. By the proof of Theorem \ref{NSthm1}, the isomorphism classes of ideals $\cI$ of $c_2=-n$ are in one to one correspondence with the isomorphism classes of stable Kronecker complexes  of the form
\begin{equation}\label{KI}
\xymatrix{
\bK_C(\cI):& \cL^\vee\otimes \CC^n\ar[r]^a &\cO_C\otimes \CC^{2n+1}\ar[r]^b & \cL^\eta\otimes \CC^n}
\end{equation}
Furthermore, by \cite[Lemma 2.6]{NS06}, we have $\cI|_C\cong H^0(\bK_C(\cI))$. For a sheaf $\cH$ on $C$, we set
$ch(\cH):=(r(\cH),d(\cH))$.
Note that $ch(H^0(\bK_C(\cI)))=(1,0)$, so we have
\[
H^0(\bK_C(\cI))\simeq \cO_C(-D)\oplus \cT, 
\]
where $D$ is divisor of a degree $l$ and $\cT$ is a torsion sheaf of length $l$. Let
\[
\cT\simeq \bigoplus_{i=1}^k\left( \bigoplus_{j=1}^{d_i}\cO_{j\cdot p_i}^{\oplus r_{ij}}\right).
\]
Note that $l=\sum_{i=1}^k\sum_{j=1}^{d_i} j\cdot r_{ij}$.
We define the underlying cycle of $\cT$ in $C^{(l)}$ by
\[
Z(\cT):=\sum_{i=1}^k (\sum_{j=1}^{d_i} jr_{ij})[p_i].
\]
It is easy to see that the endomorphism ring of $\bigoplus_{j=1}^{d}\cO_{j\cdot p}^{\oplus r_{j}}$ has dimension 
\begin{equation}\label{dimEnd}
\sum_{0\leq i, j\leq d} min(i,j) \cdot r_i r_j.
\end{equation}

Denote by $\coh_{1,0}(C)\subset \coh(C)$ the substack consisting of sheaves of rank 1 and degree 0. It has a stratification by the size of the torsion subsheaf with the strata
\[
\coh_{1,0}^{l}(C):=\{\cF\in \coh_{1,0}(C)| 
\text{length of the torsion subsheaf of $C$}=l\}.
\]
Let us set $\coh_{1,0}^{\leq l}(C)=\cup_{i=0}^l \coh_{1,0}^{i}(C)$.

By Proposition \ref{Leavescoarse}, to compute the symplectic leaves of $(\PP^2_\eta)^{[n]}$ it suffices to compute the image of the map $q:\cM^s_\eta(1,0,n)\to \coh(C)$. Here is a partial result in this direction.
Let us fix a neutral point $o\in C$ and denote by $k\cdot\eta$ the $k$th multiple of $\eta$ in the group law of $C$.

\begin{prop}\label{Hilb-leaf}
The projection $q: (\PP^2_\eta)^{[n]}\to \coh(C)$ factors through $\coh^{\leq n}_{1,0}(C)$.  Let $\cH:=\cO_C(-D)\oplus \cT$ be an object in  $\coh_{1,0}(C)$ such that $\cT$ has length $l\le n$. Assume that $F_\cH$ is non-empty and smooth.  Then $F_\cH$ has dimension at most $2n-2l$ and there is a rational equivalence of divisors $Z(\cT)-D\sim [o]-[3n\cdot\eta]$.
\end{prop}

\begin{proof}
Consider the long exact sequence of cohomology groups associated to the fiber product \eqref{FH}. The dimension of $F_\cH$ is equal to $2n+1-\dim(\Ext_C^1(\cH,\cH))$. 
The dimension calculation \eqref{dimEnd} easily implies that
$$\dim(\Ext_C^1(\cH,\cH))=\dim(\Ext^0_C(\cH,\cH))\ge 2l+1$$  
This gives our estimate on the dimension of $F_\cH$. 
(Note that if $\cT=\cO_Z$, where $Z$ is a $0$-dimensional subscheme of length $l$, 
then the dimension of $F_\cH$ is exactly $2n-2l$ if $F_\cH$ is non-empty.)

Recall that the relative moduli space $M^s_C(r,d,n)$ is smooth over the base $C$ (see Theorem \ref{NSthm1}). 
Given $\eta\in C$ and $\cO_C(-D)\oplus \cT$ in the image of $q$, we need to compute $D$. The class of $\cO_C(-D)\oplus \cT$ in the K-theory coincides with the class of the Kronecker complex \eqref{KI}. Its class is determined by
\begin{align*}
[\cO_C(-D)]+Z(\cT)&=-n[\cL^\vee]+(2n+1)[\cO_C]-n[\cL^\eta]\\
&=-n[\cL^\vee]+(2n+1)[\cO_C]-3n([\cL]+[\eta]-[o])\\
&=[\cO_C]+([o]-[3n\cdot\eta])
\end{align*}
which gives the required relation between the classes of divisors.
\end{proof}

\begin{remark}
We believe that for  $\cH:=\cO_C(-D)\oplus \cO_Z$ such that $Z-D\sim [o]-[3n\cdot\eta]$, $F_\cH$ is always non-empty and smooth.  For the case of maximal dimensional leaf ($l=0$), this was proved by Nevins-Stafford (see \cite[Cor.\ 8.10]{NS06}) and de Naeghel-Van den Bergh (see \cite[Lemma 5.2.1]{NV04}). For $l>0$, one possible way to prove the non-emptiness is to show that the Poisson deformation $(\PP^2_\eta)^{[n]}$ can be obtained by contracting a $(1,1)$-class with the Poisson bivector in the sense of Hitchin \cite{Hi12}. 
\end{remark}

\begin{remark}
Proposition \ref{Hilb-leaf} is just the first step for the classification of the symplectic leaves on $(\PP_\eta^2)^{[n]}$. The next step is to determine which sheaves can occur in the image of the map $q$ and study the geometry of the fibers. This is quite a delicate problem, which we will explore elsewhere.

As an example,
let us describe the symplectic leaves of $(\PP^2)^{[3]}$. Here $d_F$ is the symbol for the fiber dimension (i.e. the dimension of the corresponding symplectic leaf). The symplectic leaf is determined by fixing an isomorphism class of $\cI_Z|_C$.
We have the following possibilities:
\[
\cI_Z|_C=\begin{cases}
\cO_C & l=0, d_F=6 \\
\cO_C(-p)\oplus \cO_p & l=1, d_F=4\\
\cO_C(-p-q)\oplus \cO_{p\cup q} & l=2, d_F=2\\
\cO_C(-2p)\oplus\cO_p\oplus\cO_p & l=2, d_F=0\\
\cO_C(-p-q-r)\oplus \cO_{p\cup q\cup r} & l=3, d_F=0
\end{cases}
\] The $6$-dimensional leaf is isomorphic to $(\PP^2\setminus C)^{[3]}$. To describe the 2-dimensional leaf associated with
a pair of distinct points $p,q\in C$, we consider the blow up of $\PP^2$ at $p$ and $q$ and denote it by $\wh{\PP^2}_{p,q}$. Denote the exceptional fibers at $p$ and $q$ by $E_p$ and $E_q$ respectively and denote the proper transform of $C$ by $\wh{C}$. Then the 2-dimensional leaf is isomorphic to $\wh{\PP^2}_{p,q}\setminus \wh{C}$. Here points 
of $E_p\setminus \wh{C}$ 
parameterize subschemes in $\PP^2$ given as the union of a double point at $p$ (corresponding to a direction not tangent
to $C$) and $q$. 
Next, for a point $p\in C$, let $\wh{\PP^2}_p$ denote the blow up of $\PP^2$ at $p$, with the exceptional divisor $E_p$,
and let $\wh{C}$ be the proper transform of $C$. 
Then the $4$-dimensional leaf associated with $p$ is isomorphic to $(\wh{\PP^2}_p\setminus \wh{C})^{[2]}$.
More precisely, the corresponding locally closed embedding
$$(\wh{\PP^2}_p\setminus \wh{C})^{[2]}\to (\PP^2)^{[3]}$$
sends a length 2 subscheme $\wh{Z}$ of $\wh{\PP^2}_p\setminus \wh{C}$ to the subscheme $Z\subset\PP^2$  of length 3
whose ideal sheaf is obtained as the push forward of $I_{\wh{Z}}(-E_p)$.
Finally, there are two kinds of $0$-dimensional leaves: those with $l=3$ correspond to points in the symmetric product $C^{(3)}$, while those with $l=2$ correspond to points in the punctual Hilbert scheme supported on $C$. 
\end{remark}

\section{Example two: Stable triples and Feigin-Odesskii algebras}

\subsection{Poisson structures on the moduli space of stable triples}\label{triples-sec}

In this section, we prove that for the moduli space of 2-term complexes of vector bundles with appropriate stability conditions, the $0$-shifted Poisson structure constructed in Theorem \ref{Mainthm1} specializes to the Poisson structure constructed by one of us in \cite{Pol98}. For a subclass of these examples, we compute the underlying Poisson brackets explicitly and show that they coincide with the semi-classical limits of the elliptic algebras constructed by Feigin and Odesskii \cite{FO87}. To the best of our knowledge, the comparison between these Poisson structures has not appeared in the literature before.

Let $C$ be a complex elliptic curve. We will consider moduli spaces of triples $T=(V_0,V_1,\phi)$ 
consisting of vector bundles $V_i$, $i=0,1$, on $C$ and a morphism $\phi:V_0\to V_1$. 

For a real parameter $\sigma$, the $\sigma$-degree and the $\sigma$-slope of $T$ are defined by
\[
\deg_\sigma(T) =\deg~V_1+\deg~V_0+\sigma\cdot \rk(V_0), \ \ 
\mu_\sigma(T)=\frac{\deg_\sigma~T}{\rk~V_1+\rk~V_0}.
\]
A triple $T$ is called \emph{$\sigma$-stable} if for any proper subtriple $T^\prime\subset T$ on has $\mu_\sigma(T^\prime)<\mu_\sigma(T)$. We denote the moduli stack of $\sigma$-stable triples by $\cM_\sigma$. It was proved in \cite{BG96} that (the coarse moduli space of) $\cM_\sigma$ are smooth projective varieties for all $\sigma$.

It is clear that $\cM_\sigma$ is an open substack of $\dCx(C)$. Let $d$ and $r$ be positive integers such that $r<d$ and $gcd(r,d)=1$. Now we restrict to the special case:
\begin{enumerate}
\item[$\bullet$] $V_0=\cO_C$, $V_1$ is a rank $r+1$ vector bundle and $\sigma=2d/r$.
\end{enumerate}

From \cite[Sec.\ 3]{Pol98} we know that in this case stable triples are precisely those for which $\phi: \cO_C\to V_1$ is an embedding of a subbundle and
$V_1/\phi(\cO_C)$ is a stable vector bundle (of rank $r$ and degree $d$). Let us denote this moduli space by $\cN_{r+1,d}$, and let $\cU_{r,d}$ be the moduli stack of stable vector bundles of rank $r$ and degree $d$ on $C$. We have a map $q: \cN_{r+1,d}\to \cU_{r,d}$ sending $\{\xymatrix{\cO_C\ar[r]^\phi & V_1}\}$ to $V_1/\phi(\cO_C)$. It is easy to check that $q$ is a smooth map. 
We denote by $\cN_\xi$ the fiber of $q$ over $\xi\in \cU_{r,d}$. 
Since $\cN_\xi$ is obtained by fixing the first term of a complex, together with its quasi-isomorphism class,
by Corollary \ref{Casi}, $\cN_\xi$ carries a $0$-shifted Poisson structure.

Note that $\cN_\xi$ is a $\GG_m$-gerbe over its coarse moduli space $N_\xi$ which is isomorphic to 
$\PP(\Ext^1(\xi_{r,d},\cO_C))$. Similarly to Proposition \ref{Poicoarse} we see that 
the classical shadow of the $0$-shifted Poisson structure on $\cN_\xi$ descends to $N_\xi$.

\begin{theorem}\label{stable-triples-thm}
The classical shadow of the $0$-shifted Poisson structure on $\cN_\xi$ coincides with the Poisson structure defined in \cite{Pol98}.
\end{theorem}
\begin{proof}
In the proof of Theorem \ref{Thmapp1} we have calculated that the morphism $\Pi_h$ associated with our $0$-shifted Poisson structure is equal to $\ad\circ \ft$.
This is precisely the chain map that induces the Poisson structure in \cite{Pol98} (see \cite[Sec.\ 6]{Pol98}).
\end{proof}

\subsection{Semi-classical limit of Feigin-Odesskii algebras}\label{FO-sec}

Now we are going to recall the definition of the elliptic algebras due to Feigin and Odesskii \cite{FO87}. We will show that in the case when $\xi$ is a line bundle, the Poisson structure on $N_\xi$ coincides with the classical limit of a class of elliptic algebras.

Let $C\cong\CC/\Gamma$, where $\Gamma=\ZZ+\ZZ\tau$ and $\eta\in\CC$. 
We denote the group of $n$-torsion points of $C$ by 
\[
\Gamma_n = \left(\frac{1}{n}\Gamma\right)/\Gamma = \left\{\frac{a_1}{n} + \frac{a_2}{n} \tau + \Gamma: \;\; a_1, a_2 \in \ZZ\right\} \subset C,
\]
 
Consider the function
\begin{equation}\label{eq-delta}
\zeta(z)= -e^{-2 \pi \ii(z-b)}, \hs z \in \CC,
\end{equation}
where $b \in \CC$ is such that $b\equiv\frac{(n-1)~\tau}{2n}+ \frac{c}{n}\mod \frac{1}{n}\ZZ$. 
The $\GG_m$-multiplier
\[
e_1(z):=1 ~ \mbox{and} ~ e_{\tau/n}(z):=\zeta(z)
\] defines a line bundle $L_{\zeta}$ on $\CC/\ZZ+\ZZ\frac{\tau}{n}$ of degree $1$. Let us denote its pull back to $C$ by $L_{n,c}$.

For positive integers $n$ and $k$ such that $0<k<n$ and $\gcd(n,k)=1$,
Feigin and Odesskii defined in \cite{FO87} a family of quadratic algebras $Q_{n,k}(C,\eta)$ over $\CC$. Here $\eta$
is a complex parameter, which we should view as defining a point on the elliptic curve $C$.
The degree 1 piece of $Q_{n,k}(C,\eta)$ is the space $\Theta_{n,c}$ of global holomorphic sections of $L_{n,c}$. 
By definition, $\Theta_{n, c}$  consists of
holomorphic functions $f$ on $\CC$ satisfying
\[
f(z+1) = f(z), \hs f(z+\tau) = (-1)^{n} e^{-2 \pi \ii (nz-c)} f(z), \hs z \in \CC.
\]

Let $H_n$ denote the Heisenberg group of order $n^3$, which is the
group with generators $h_1, h_2, \epsilon$ and relations
\[
h_1h_2=\ep h_2h_1,\hs h_1\ep=\ep h_1, \hs h_2\ep=\ep h_2, \hs  h_1^n=h_2^n=\ep^n=1.
\]
Define two operators $T_{1/n}$ and $T_{\tau/n}$ on the space of $\CC$-valued functions on $\CC$ by 
\[
(T_{1/n}~f)(z)=f(z+\frac{1}{n}),\hs (T_{\tau/n}~f)(z)=\zeta(z)^{-1}~f(z+\frac{\tau}{n}).
\]
It is easy to check that $\Theta_{n,c}$ is invariant under
$T_{1/n}$ and $T_{\tau/n}$, and that  
\[
T_{1/n}^n=T_{\tau/n}^n=1 \hs \mbox{and} \hs T_{1/n}T_{\tau/n}=e^{\frac{2\pi\ii}{n}}~T_{\tau/n}T_{1/n}.
\]
Thus, the assignment 
\begin{equation}\label{eq-Hn-Theta}
h_1\longmapsto T_{1/n}, \hs h_2\longmapsto T_{\tau/n}
\end{equation}
defines a representation of $H_{n}$ on $\Theta_{n, c}$, in which 
$\epsilon \in H_{n}$ acts by the scalar multiplication by $\omega=\omega_{n}:= e^{\frac{2\pi\ii}{n}}$.

Note that the theta function
\[
\theta(z)=\sum_{n\in\ZZ} (-1)^n\cdot e^{2\pi \ii (nz+\frac{n(n-1)}{2}\tau)}
\] is an element in $\Theta_{1,0}$. It is easy to check that 
\[
\theta_\alpha(z):=\theta(z+\frac{\alpha}{n}\tau)\theta(z+\frac{1}{n}+\frac{\alpha}{n}\tau)\ldots\theta(z+\frac{n-1}{n}+\frac{\alpha}{n}\tau)\cdot e^{2\pi\ii(\alpha z+\frac{\alpha(\alpha-n)}{2n}+\frac{\alpha}{2n})}
\] form a canonical basis for $\Theta_{n,\frac{n-1}{2}}$. We will use the following properties of functions $\theta_\alpha$ (see Appendix of \cite{Od03}):
\begin{enumerate}
\item[(1)] $\theta_\alpha(z+\frac{1}{n})=e^{2\pi\ii\alpha/n} \theta_\alpha(z)$;
\item[(2)] $\theta_\alpha(z+\frac{\tau}{n})=e^{-2\pi\ii(z+\frac{1}{2n}-\frac{n-1}{2n}\tau)} \theta_{\alpha+1}(z)$;
\item[(3)] $\theta_{-\alpha}(-z)=-e^{-2\pi\ii\alpha/n}e^{-2\pi\ii nz}\theta_\alpha(z)$.
\end{enumerate}
The first two properties are equivalent to the following formulas for the $H_n$-action:
\begin{align}\label{automorphy-theta}
T_{1/n}\theta_\alpha=e^{\frac{2\pi\ii\alpha}{n}}~\theta_\alpha, \hs T_{\tau/n}\theta_\alpha=\theta_{\alpha+1}, \hs \alpha 
\in \ZZ/n\ZZ. 
\end{align}

The \emph{Feigin-Odesskii algebra} $Q_{n,k}(C,\eta)$ is defined to be the quotient of the free algebra $\CC\lg x_i: i\in\ZZ/n\ZZ\rg$ by the quadratic relations
\begin{align}\label{FOrelation}
\sum_{r\in\ZZ/n\ZZ} \frac{\theta_{j-i+r(k-1)}(0)}{\theta_{kr}(\eta)\theta_{j-i-r}(-\eta)}x_{j-r}x_{i+r}.
\end{align}
In the limit as $\eta\to 0$, we get the polynomial algebra $\CC[x_i:i\in\ZZ/n\ZZ]$.
The \emph{semi-classical limit} of $Q_{n,k}(C,\eta)$, denoted by $q_{n,k}$, is this polynomial algebra equipped with the Poisson bracket $\{x_i,x_j\}:=\lim_{\eta\to 0}\frac{[x_i,x_j]}{\eta}$. It follows from the relations \eqref{FOrelation} that for $i\neq j$,
\begin{align}\label{FObracket}
\{x_i,x_j\}=\left(\frac{\theta^\prime_{j-i}(0)}{\theta_{j-i}(0)}+\frac{\theta^\prime_{k(j-i)}(0)}{\theta_{k(j-i)}(0)}-2\pi\ii n\right)x_ix_j+
\sum_{r\neq 0,j-i} \frac{\theta_{j-i+r(k-1)}(0)\theta^\prime_0(0)}{\theta_{kr}(0)\theta_{j-i-r}(0)} x_{j-r}x_{i+r}.
\end{align}
Because the bracket is quadratic, it defines a Poisson structure on the projective space $\PP^{n-1}$. 
We call this bracket the \emph{Sklyanin bracket}.

\begin{theorem}\label{qn1=ext}
Let $\xi$ be the line bundle $L_{n,\frac{n+1}{2}}$ defined above.
There is an isomorphism of Poisson varieties between $(N_\xi,H^0(\Pi_h))$ and $\PP^{n-1}$ equipped with the Sklyanin bracket coming from $q_{n,1}$.
\end{theorem}

The proof of Theorem \ref{qn1=ext} will take up the rest of this section.

For brevity we denote the Poisson bivector  $H^0(\Pi_h)$ on $N_\xi$ by $\pi$.
Let $(U_+,U_-)$ be an open affine covering of $C$. We will compute the Poisson bracket associated to $\pi$ using the Cech complex and directly compare the result with the Sklyanin bracket \eqref{FObracket}.

Let $t$ be a class in $\Ext^1(\xi,\cO_C)$ and let
\[
\xymatrix{
0\ar[r] & \cO_C\ar[r]^{s} & V_t\ar[r]^{a} &\xi\ar[r] & 0
}
\] be the corresponding extension. By an abuse of notation, we use the same symbol $t$ to denote the corresponding point in $N_\xi$.
Let $\sEnd(V_t,\cO_C)$ denote the sheaf of endomorphisms of $V_t$ preserving $\cO_C$. 
We have an identification of the tangent space at $t$ to the moduli space of triples with $H^1(C,\sEnd(V_t,\cO_C))$ (see
\cite[Lem.\ 3.1]{Pol98}), and the morphism
$\Pi_h$ inducing the Poisson structure can be viewed as a morphism 
$\sEnd(V_t,\cO_C)^\vee\to \sEnd(V_t,\cO_C)[-1]$ in the derived category, represented by the chain map
\[
\xymatrix{
 \xi^\vee \ar[r]^{-d^*} \ar[d]^{d^*} & \sEnd(V_t)\ar[d]^{0}\\
\sEnd(V_t)\ar[r]^d& \xi
}
\]
where $d(A)=a\circ A\circ s$, $d^*(\psi)=s\circ\psi\circ a$.
The natural exact sequence
\[
\xymatrix{
0\ar[r] &\xi^\vee\ar[r] & \sEnd(V_t,\cO_C)\ar[r] & \sEnd~\xi\oplus\cO_C\ar[r] &0
}
\]

leads to the following identification of the tangent space $T_t N_\xi$:
\begin{align*}
T_t N_\xi &\simeq \ker\Bigg(H^1(C,\sEnd(V_t,\cO_C))\to H^1(C,\sEnd~\xi)\oplus H^1(C,\cO_C)\Bigg)\\
&\simeq \Cok\Bigg(\End(\xi)\oplus H^0(C,\cO_C)\to H^1(C,\xi^\vee)\Bigg)\\
&\simeq\Cok\Bigg(t_*:H^0(C,\cO_C)\to H^1(C,\xi^\vee)\Bigg)
\end{align*}
where the last equality is due to the fact that the maps $\End(\xi)\to H^1(C,\xi^\vee)$ and $H^0(C,\cO_C)\to H^1(C,\xi^\vee)$, induced by the extension class $t$, differ only by sign.
Dually, the cotangent space $T^*_t N_\xi$ is isomorphic to 
\[
\ker\Bigg(t^*:H^0(C,\xi)\to H^1(C,\cO_C)\Bigg).
\]
We will give a formula for the map $\pi_t:T^*_t N_\xi\to T_t N_\xi$ on $\phi\in \ker(t^*)\subset H^0(C,\xi)$.

We need to lift $\phi\in H^0(C,\xi)$ to an element of the hypercohomology $\HH^1(\xi^\vee\to \sEnd~V_t)$.
Such a lifting is represented by a Cech $1$-cocycle $(\psi_\pm;A_+,A_-)$, where $A_+\in \sEnd~V_t(U_+)$,
$A_-\in \sEnd~V_t(U_-)$, $\psi_\pm\in \xi^\vee(U_+\cap U_-)$, and 
\begin{equation}\label{hypercoh-coc-A-eq}
-s\psi_\pm a=A_+-A_-
\end{equation}
over $U_+\cap U_-$.
The class of $(\psi_\pm;A_+,A_-)$ lifts $\phi$ if 
$$aA_+s=aA_-s=\phi.$$ 
Let $(t_+,t_-,b_+,b_-)$ be local splittings (defined over $U_+$ and $U_-$) of the short exact sequence
\[
\xymatrix{
0\ar[r]&\cO_C\ar[r]^{s} & V_t\ar[r]^a\ar@/^1pc/[l]_{t_\pm} & \xi\ar@/^1pc/[l]_{b_\pm}\ar[r]&0
}
\]
where $s\circ t_\pm+b_\pm\circ a=\id_{V_t}$.
Note that we have 
$$b_+-b_-=s\psi_t$$ 
for some $\psi_t\in \xi^\vee(U_+\cap U_-)$ representing the class $t\in H^1(\xi^\vee)$. It is easy to see that then
$$t_+-t_-=-\psi_t a.$$
The function
$$f:=(\psi_t,\phi)\in \cO(U_+\cap U_-)$$
represents the class $t^*(\phi)=0$, where $(~,~)$ is the natural pairing between $\xi$ and $\xi^\vee$. Hence, there exists $f_+\in \cO(U_+)$ and $f_-\in \cO(U_-)$ such that
$$f=f_+-f_-.$$ 
Let us set $\wt{A}_+=b_+\phi t_+$ and $\wt{A}_-=b_-\phi t_-$. 
Then we have 
$$\wt{A}_+-\wt{A}_-=b_+\phi t_+-b_-\phi t_-.$$
Hence,
$$a(\wt{A}_+-\wt{A}_-)=-\phi\psi_t a=-fa, \ \ (\wt{A}_+-\wt{A}_-)s=s\psi_t\phi=sf.$$
This gives the way to correct $\wt{A}_+$ and $\wt{A}_-$: setting 
$$A_+=\wt{A}_+ + f_+b_+a - f_+st_+, \ \ A_-=\wt{A}_- + f_-b_-a - f_-st_-,$$
we now check that 
\begin{align*}
&A_+-A_-=(b_-+s\psi_t)\phi(t_- - \psi_ta) -b_-\phi t_-
+(f_- +f)(b_- +s\psi_t)a-f_-b_-a \\
& -(f_- +f)s(t_--\psi_ta)+f_-st_-
=s(2f_-+f)\psi_ta.
\end{align*}
Hence, \eqref{hypercoh-coc-A-eq} holds with
$$\psi_\pm=-(2f_-+f)\psi_t.$$
Note that the image of $(\psi_\pm;A_+,A_-)$ in $\HH^1(\sEnd~V_t\to \xi)$ is the class $(s\psi_\pm a;0,0)$.
Hence, $\pi_t(\phi)\in\Cok(t_*)$ is represented by  the Cech $1$-cocycle $\psi_\pm$.
So we get the formula for $\pi_t$:
\begin{align}\label{poisson1.2}
\pi_t(\phi)=-(2f_-+f)\psi_t.
\end{align}
The kernel of $\cO_C(U_+\cap U_-)\to H^1(\cO_C)$, denoted by $\cO_C(U_+\cap U_-)^0$  is the subspace of functions with zero residue. Let $P_-$ (resp. $P_+$) be the projection $\cO_C(U_+\cap U_-)^0\to \xi(U_-)\cong\cO_C(U_-)$ (resp. $\cO_C(U_+\cap U_-)^0\to\xi^\vee(U_+)\cong\cO(U_+)$ ). The projection is well defined up to the addition of a constant. We may set $f_-= P_-(f)$ and $f_+=P_+(f)$. Then formula \eqref{poisson1.2} can be rewritten as
\begin{align}\label{poisson1.3}
\pi_t(\phi)=-(2P_-(\psi_t\phi)+\psi_t\phi)\psi_t=(\psi_t\phi-2P_+(\psi_t\phi))\psi_t.
\end{align}
For a different choice of constant in the definition of $P_-$ (or $P_+$) the formula will differ by a constant multiple of $\psi_t$, therefore defines the same Poisson structure on $N_\xi$.

Now set $\xi=L_{n,\frac{n+1}{2}}$.
As an open covering of $C$ we take $(U_+,U_-)$, where $U_+=C\setminus D$ with $D:=\{\frac{i}{n}:i=0,\ldots,n-1\}$, and 
$U_-$ is the union of formal discs centered at
$z=\frac{i}{n}$ for $i=0,1,\ldots,n-1$. Note that $D$ is precisely the zero divisor of $\theta_0$, which is a section
of $L_{n,\frac{n+1}{2}}$. Below we always use an isomorphism 
\[ 
\xymatrix{
L_{n,\frac{n+1}{2}}\ar[r]^{\sim} & \cO(D): s\mapsto \frac{s}{\theta_0}.
}
\]
Under this isomorphism the basis $(\theta_\alpha)$ of $H^0(L_{n,\frac{n+1}{2}})$ maps to the functions
$$\phi_\alpha:=\frac{\theta_\alpha}{\theta_0}.$$

An element $g\in \cO(U_+\cap U_-)$ is a vector of Laurent series $(g_i)_{i=0}^{n-1}$ with $g_i\in \CC((z))$. We define a bilinear form on $\cO(U_+\cap U_-)$ by 
\[
\lg f,g\rg:=\tr(fg), \ \text{ where } \tr(f):=\frac{1}{n}\sum_{i=0}^{n-1}\Res_{z=\frac{i}{n}} f dz.
\]
Fixing the standard $1$-form 
$dz$ on $C$, an element $\psi\in H^1(\cO(-D))$ can be represented by a vector in $\cO(U_+\cap U_-)$.
The induced pairing between $\phi\in\cO(U_+)$ and $\psi\in\cO(U_+\cap U_-)$, given by
$$\lg \phi,\psi\rg:=\lg \phi|_{U_+\cap U_-},\psi\rg$$
descends to a perfect pairing between $H^0(\cO(D))$ and $H^1(\cO(-D))$.

Let us define elements $\psi_\alpha\in H^1(\cO(-D))$, for $\alpha\in\ZZ/n\ZZ$, by
$$\psi_\alpha:=(\frac{\theta_0^\prime(0)}{\theta_\alpha(i/n)})_i \ \text{ for }\ \alpha\in\ZZ/n\ZZ\setminus 0, \
\ \psi_0:=(\frac{1}{z-i/n})_i.$$  
It is easy to see that $(\psi_\alpha)_{\alpha\in\ZZ/n\ZZ}$ is a basis for $H^1(\cO(-D))$. 
Furthermore, this basis is dual to the basis $(\phi_\alpha)$ of $H^0(\cO(D))$ with respect to the above pairing. 
Indeed, for $\beta\neq 0$ we have 
\begin{align*}
\lg \phi_\alpha,\psi_\beta\rg&=\frac{1}{n}\sum_{k=0}^{n-1} \Res_{z=0} \frac{\theta_\alpha(z+k/n)\theta_0^\prime(0)}{\theta_0(z+k/n)\theta_\beta(k/n)}\\
&=\frac{1}{n}\sum_{k=0}^{n-1} \omega^{k(\alpha-\beta)} \Res_{z=0} \frac{\theta_\alpha(z)\theta_0^\prime(0)}{\theta_0(z)\theta_\beta(0)}=\delta_{\alpha\beta},\\
\end{align*}
where  $\omega=e^{2\pi\ii/n}$. Similarly,
$$\lg \phi_\alpha,\psi_0\rg=\frac{1}{n}\sum_{k=0}^{n-1} \Res_{z=0} \frac{\theta_\alpha(z+k/n)}{\theta_0(z+k/n)z}
=\frac{1}{n}\sum_{k=0}^{n-1} \omega^{k\alpha} \Res_{z=0} \frac{\theta_\alpha(z)}{\theta_0(z)z}=\delta_{\alpha 0}.$$

Let $(x_0,\ldots,x_{n-1})$ be the coordinates on $H^1(\cO(-D))$ corresponding to the basis $(\psi_\alpha)$. 
We are going to write a formula for our Poisson bracket on the open subset $x_0\neq 0$ of the projective space
$\PP H^1(\cO(-D))$, in terms of the standard coordinates $t_1,\ldots,t_{n-1}$, where $t_i=x_i/x_0$. 
We also set $t_0=1$.
Below we always identify the indices with elements of $\ZZ/n\ZZ$.
Let $\psi_t=\sum_{c\in\ZZ/n\ZZ} t_c \psi_c$ be an element in $H^1(\cO(-D))$ over this open subset.
Then the differentials of $t_i$, $i=1,\ldots,n-1$ correspond to the basis $(\phi_i-t_i\phi_0)$ of the cotangent space to $\psi_t$.
Thus, from \eqref{poisson1.3} we get for $i\neq 0$, $j\neq 0$,
$$\{t_i,t_j\}=\tr\bigl(\bigl[\psi_t(\phi_i-t_i\phi_0)-2P_+[\psi_t(\phi_i-t_i\phi_0)]\bigr]\cdot \psi_t(\phi_j-t_j\phi_0)\bigr).$$
Using the fact that $\{t_i,t_j\}$ is skew-symmetric we can rewrite this as follows:
\begin{equation}\label{ti-tj-Pplus-bracket}
\{t_i,t_j\}=\tr\bigl(P_+[\psi_t(\phi_j-t_j)]\cdot\psi_t(\phi_i-t_i)-P_+[\psi_t(\phi_i-t_i)]\cdot\psi_t(\phi_j-t_j)\bigr).
\end{equation}
To rewrite this further we need an explicit formula for $P_+(\psi_\alpha\phi_\beta)$.

\begin{lemma}\label{Pplus}
For $i\neq 0$, $j\neq 0$ and $i\neq j$, one has
\begin{equation}\label{P+psi-j-phi-i-eq}
P_+(\psi_j\phi_i)=(\frac{\theta^\prime_0(0)\theta_i(0)}{\theta_j(0)\theta_{i-j}(0)})\cdot\frac{\theta_{i-j}(z)}{\theta_0(z)}.
\end{equation}
For $j\neq 0$ one has 
$$P_+(\psi_j\phi_0)=0.$$
Also, for $i\neq 0$, one has
\begin{equation}\label{P+psi-0-phi-i-eq}
P_+(\psi_0\phi_i)=-(\frac{\theta_i(z)}{\theta_0(z)})^\prime+
\left[\frac{\theta^\prime_i(0)}{\theta_i(0)}-\pi i n\right]\cdot
\frac{\theta_i(z)}{\theta_0(z)}. 
\end{equation}
Finally, for any linear combination $\sum c_i\psi_i\phi_i$ with $\sum_i c_i=0$, one has
$$P_+(\sum c_i\psi_i\phi_i)=0.$$
\end{lemma}

\begin{proof}
Recall that we have
$$\theta_j(z+\frac{i}{n})=\omega^{ij}\theta_j(z),$$
where $\omega=e^{2\pi\ii/n}$.
To prove \eqref{P+psi-j-phi-i-eq} we have to check that for $k=0,\ldots,n-1$, the function
$$\frac{\theta'_0(0)}{\theta_j(k/n)}\cdot \frac{\theta_i(z)}{\theta_0(z)}-
(\frac{\theta^\prime_0(0)\theta_i(0)}{\theta_j(0)\theta_{i-j}(0)})\cdot\frac{\theta_{i-j}(z)}{\theta_0(z)}$$
is regular near $z=k/n$. But this function at most has pole of order $1$ and its residue at $z=k/n$ is equal to
\begin{align*}
&\frac{\theta'_0(0)}{\theta_j(k/n)}\cdot \frac{\theta_i(k/n)}{\theta'_0(k/n)}-
(\frac{\theta^\prime_0(0)\theta_i(0)}{\theta_j(0)\theta_{i-j}(0)})\cdot\frac{\theta_{i-j}(k/n)}{\theta'_0(k/n)}=\\
&\frac{\theta'_0(0)}{\omega^{jk}\theta_j(0)}\cdot \frac{\omega^{ik}\theta_i(0)}{\theta'_0(k/n)}-
(\frac{\theta^\prime_0(0)\theta_i(0)}{\theta_j(0)})\cdot\frac{\omega^{(i-j)k}}{\theta'_0(k/n)}=0.
\end{align*}

The vanishing of $P_+(\psi_j\phi_0)=P_+(\psi_j)$ for $j\neq 0$ is clear since $\psi_j$ is regular on $U_-$.
To check \eqref{P+psi-0-phi-i-eq} we need to show that the difference between $\psi_0\phi_i$ and the right-hand side
is regular near each $z=k/n$. Since $\phi_i(z+k/n)=\omega^{ik}\phi_i(z)$, it is enough to consider $z=0$.
The Laurent expansion of $\phi_i$ near $z=0$ has form
$$\phi_i(z)=\frac{\theta_i(z)}{\theta_0(z)}=
\frac{\theta_i(0)}{\theta'_0(0)}\cdot \frac{1}{z}
\left[\frac{\theta'_i(0)}{\theta'_0(0)}-\frac{\theta_i(0)\theta''_0(0)}{2\theta'_0(0)^2}\right]+\ldots$$
Using property (3) of $\theta$-functions, one can check that 
$$\frac{\theta_0^{\prime\prime}(0)}{\theta_0^{\prime}(0)}=2\pi\ii n.$$
Hence, we can rewrite the above expansion as
\begin{equation}\label{phi-i-L-exp}
\phi_i(z)=\frac{\theta_i(0)}{\theta'_0(0)}\cdot \frac{1}{z}
+\left[\frac{\theta'_i(0)}{\theta'_0(0)}-\pi in\cdot \frac{\theta_i(0)}{\theta^\prime_0(0)}\right]+\ldots
\end{equation}
The expansion of $\psi_0\phi_i$ near $z=0$ is obtained from this by multiplication with $1/z$.

On the other hand, the right-hand side of \eqref{P+psi-0-phi-i-eq} has the expansion at $z=0$,
$$\frac{\theta_i(0)}{\theta'_0(0)}\cdot \frac{1}{z^2}+
\left[\frac{\theta^\prime_i(0)}{\theta_i(0)}-\pi \ii n\right]\cdot
\frac{\theta_i(0)}{\theta'_0(0)}\cdot \frac{1}{z}+\ldots.$$ 
It follows that we do get the same polar parts as for $\psi_0\phi_i$.

Finally, to prove the last property it is enough to check that for $i\neq 0$,
$$P_+(\psi_i\phi_i-\psi_0\phi_0)=P_+(\psi_i\phi_i-\psi_0)=0.$$
But $\psi_i\phi_i-\psi_0$ has at most pole of order $1$ and
$$\Res_{z=k/n}(\psi_i\phi_i)=\frac{\theta'_0(0)}{\theta_i(k/n)}\cdot \frac{\theta_i(k/n)}{\theta'_0(k/n)}=1=\Res_{z=k/n}\psi_0,$$
since $\theta'_0(k/n)=\theta'_0(0)$.
\end{proof}

For $\alpha,\beta\in\Z/n\Z$, such that $\alpha\neq 0$ and $\beta\neq 0$, let us set
\begin{equation}\label{F-def-eq}
F(\alpha,\beta):=\frac{\theta^\prime_0(0)\theta_{\alpha+\beta}(0)}{\theta_\alpha(0)\theta_\beta(0)}, \ \ \ 
F(0,\alpha)=F(\alpha,0):=\frac{\theta^\prime_\alpha(0)}{\theta_\alpha(0)}-\pi \ii n, \ \ \ F(0,0)=0.
\end{equation}
Then by Lemma \ref{Pplus}, with this notation
we have for $i\neq 0$,
\begin{align*}
&P_+[\psi_t(\phi_i-t_i\phi_0)]=\sum_{\alpha\neq i}t_\alpha P_+[\psi_\alpha\phi_i]-
t_i\sum_{\alpha\neq 0}t_\alpha P_+[\psi_\alpha\phi_0]+t_iP_+[\psi_i\phi_i-\psi_0\phi_0]\\
&=\sum_{\alpha\neq i}t_\alpha P_+[\psi_\alpha\phi_i]=\sum_{\alpha\neq i}t_\alpha F(\alpha,i-\alpha)\phi_{i-\alpha}-\phi^\prime_i.
\end{align*}

Next, we observe that the functional $\tr$ on $\cO(U_+\cap U_-)$ is invariant with respect to the action of the generator
$h_1$ of the Heisenberg group that acts by the shift by $1/n$. Since $h_1\theta_\alpha=\omega^\alpha\theta_\alpha$,
we deduce that 
$$h_1\phi_\alpha=\omega^\alpha\phi_\alpha, \ \ 
h_1\psi_\alpha=\omega^{-\alpha}\phi_\alpha, \ \ h_1\phi^\prime_\alpha=\omega^\alpha\phi^\prime_\alpha.$$
Thus, we have the $\Z/n\Z$-weights $wt(\phi_\alpha)=wt(\phi'_\alpha)=-wt(\psi_\alpha)=\alpha$, and $\tr$ kills expressions
of nonzero weight. Thus, using the above computation of $P_+$ we can write for $i\neq 0$,
\begin{align*}
&\tr\bigl(P_+[\psi_t(\phi_i-t_i)]\cdot\psi_t(\phi_j-t_j)\bigr)\\
&=\sum_{\alpha\neq i}t_\alpha F(\alpha,i-\alpha)\tr(\phi_{i-\alpha}(\phi_j-t_j)\psi_t)-\tr(\phi^\prime_i(\phi_j-t_j)\psi_t)\\
&=
\sum_{\alpha\neq i}t_\alpha t_{i+j-\alpha} F(\alpha,i-\alpha)\tr(\phi_{i-\alpha}\phi_j\psi_{i+j-\alpha})
-t_j\sum_{\alpha\neq i}t_\alpha t_{i-\alpha} F(\alpha,i-\alpha)\\
&-t_{i+j}\tr(\phi^\prime_i\phi_j\psi_{i+j}),
\end{align*}
where we used the identities $\tr(\phi_{i-\alpha}\psi_{i-\alpha})=1$ and $\tr(\phi^\prime_i\psi_i)=0$.
Plugging this into \eqref{ti-tj-Pplus-bracket} we can rewrite our Poisson bracket as
\begin{align*}
&\{t_i,t_j\}=\\
&\sum_{\alpha\neq j}t_\alpha t_{i+j-\alpha} F(\alpha,j-\alpha)\tr(\phi_{j-\alpha}\phi_i\psi_{i+j-\alpha})
-\sum_{\alpha\neq i}t_\alpha t_{i+j-\alpha} F(\alpha,i-\alpha)\tr(\phi_{i-\alpha}\phi_j\psi_{i+j-\alpha})\\
&-t_i\sum_{\alpha\neq j}t_\alpha t_{j-\alpha} F(\alpha,j-\alpha)+
t_j\sum_{\alpha\neq i}t_\alpha t_{i-\alpha} F(\alpha,i-\alpha)\\
&+t_{i+j}[-\tr(\phi^\prime_j\phi_i\psi_{i+j})+\tr(\phi^\prime_i\phi_j\psi_{i+j})].
\end{align*}
Changing the summation variable in the first sum by $\alpha=j-r$ and in the second sum by $\alpha=i+r$, 
we can rewrite this as
\begin{align}\label{ti-tj-F-tr-eq}
&\{t_i,t_j\}= \nonumber\\
&\sum_{r\neq 0}t_{j-r}t_{i+r} F(j-r,r)\tr(\phi_r\phi_i\psi_{i+r})
-\sum_{r\neq 0}t_{i+r} t_{j-r} F(i+r,-r)\tr(\phi_{-r}\phi_j\psi_{j-r}) \nonumber\\
&-t_i\sum_{r\neq j}t_r t_{j-r} F(r,j-r)+
t_j\sum_{r\neq i}t_r t_{i-r} F(r,i-r) \nonumber\\
&+t_{i+j}[-\tr(\phi^\prime_j\phi_i\psi_{i+j})+\tr(\phi^\prime_i\phi_j\psi_{i+j})].
\end{align}

The next important observation is that since our bracket on the moduli space $N_\xi$ is given by the natural construction,
it is preserved by the action of the Mumford group of the line bundle $\xi$, which acts on the elliptic curve and on $\xi$,
hence on $N_\xi=\PP\Ext^1(\xi,\cO)$.
It follows that our bracket on the projective space can be lifted to an $H_n$-invariant quadratic Poisson bracket on
the affine space $\Ext^1(\xi,\cO)$. Indeed, it is well-known that the projection from the space of quadratic Poisson brackets on
the affine space to the space of Poisson brackets on the projective space is surjective (see \cite{Bondal}, \cite[Sec.\ 12]{Pol97}).
Since this projection is linear, the induced map between the subspaces of $H_n$-invariants is still surjective.
Now we use the following simple general statement.

\begin{lemma} Let $\{\cdot,\cdot\}$ be a quadratic Poisson bracket on the affine space with coordinates $(x_i)$, $i\in\Z/n\Z$,
which is $H_n$-invariant. Then there exists a unique set of constants $C(\alpha,\beta)$, $\alpha,\beta\in\ZZ/n\ZZ$, such that
$$\{x_i,x_j\}=\sum_{r\in\Z/n\Z} C(r,j-i-r)x_{i+r}x_{j-r},$$
$$C(\beta,\alpha)=C(\alpha,\beta)=-C(-\alpha,-\beta).$$
The corresponding Poisson bracket on the projective space is given by
\begin{align}\label{H-inv-proj-Pois-eq}
&\{t_i,t_j\}=\sum_{r\neq 0,j-i}C(r,j-i-r)t_{i+r}t_{j-r} \nonumber\\
&-t_i\sum_{r\neq 0,j}C(r,j-r)t_rt_{j-r}-t_j\sum_{r\neq 0,-i}C(r,-i-r)t_{i+r}t_{-r} \nonumber\\
&+2[C(0,j-i)-C(0,j)-C(0,-i)]t_it_j,
\end{align}
where $t_i=x_i/x_0$ are functions on the open affine subset $x_0\neq 0$.
Here we grouped the terms in such a way that for $i\neq j$ the sets of monomials in different groups do not intersect.
\end{lemma}

\begin{proof}
The invariance with respect to $h_1$ means that
$$\{x_i,x_j\}=\sum_{r\in\Z/n\Z} C_r(i,j) x_{i+r}x_{j-r},$$
for some uniquely determined constants $C_r(i,j)$ such that $C_r(i,j)=C_{j-i-r}(i,j)$.
Furthermore, the skew-symmetry is equivalent to the identity $C_r(i,j)=-C_{-r}(j,i)$. 
Now the invariance with respect to $h_2$ gives $C_r(i,j)=C_r(i+1,j+1)$, i.e., $C_r(i,j)$ depends only on the difference $j-i$.
Thus, we can write $C_r(i,j)=C(r,j-i-r)$, which gives the first assertion.
The second assertion is obtained directly from the formula
$$\{t_i,t_j\}=\{\frac{x_i}{x_0},\frac{x_j}{x_0}\}=\frac{\{x_i,x_j\}}{x_0^2}-t_i\cdot \frac{\{x_0,x_j\}}{x_0^2}-
t_j\cdot\frac{\{x_i,x_0\}}{x_0^2}.$$
\end{proof}

For example, for the Sklyanin bracket \eqref{FObracket} with $k=1$ we have
$$C(\alpha,\beta)=F(\alpha,\beta),$$
where $F$ is defined by \eqref{F-def-eq}.

Let us denote by $C(\alpha,\beta)_M$ the constants corresponding to some $H_n$-invariant lifting of our Poisson bracket
on $N_\xi$. Then by looking at the coefficient of $t_it_rt_{j-r}$, where $r\neq 0,j$ in \eqref{ti-tj-F-tr-eq} we immediately see that 
\begin{equation}\label{C-M-F-eq}
C(\alpha,\beta)_M=F(\alpha,\beta) \ \text{ for } \alpha\neq 0,\beta\neq 0, \alpha+\beta\neq 0.
\end{equation}
Similarly, for $i\neq j$, looking at the coefficient of $t_it_j$ in \eqref{ti-tj-F-tr-eq} we get
\begin{align}\label{C-F-tr-eq}
&2[C(0,j-i)_M-C(0,j)_M-C(0,-i)_M] \nonumber\\
&=F(i,j-i)\tr(\phi_{j-i}\phi_i\psi_j)-F(j,i-j)\tr(\phi_{i-j}\phi_j\psi_i)-F(0,j)+F(0,i).
\end{align}
Now using the Laurent expansion \eqref{phi-i-L-exp}, we easily find
\begin{align*}
&\tr(\phi_{j-i}\phi_i\psi_j)=
\frac{\theta'_0(0)}{\theta_j(0)}\times\\
&\left(\frac{\theta_{j-i}(0)}{\theta'_0(0)}\cdot\left[\frac{\theta'_i(0)}{\theta'_0(0)}-\pi i n\cdot\frac{\theta_i(0)}{\theta'_0(0)}\right]+
\frac{\theta_{i}(0)}{\theta'_0(0)}\cdot\left[\frac{\theta'_{j-i}(0)}{\theta'_0(0)}-\pi i n\cdot\frac{\theta_{j-i}(0)}{\theta'_0(0)}\right]\right)\\
&=\frac{1}{\theta'_0(0)\theta_j(0)}\cdot \left(\theta_{j-i}(0)\theta'_i(0)+
\theta_{i}(0)\theta'_{j-i}(0)-2\pi \ii n \theta_{j-i}(0)\theta_{i}(0)\right).
\end{align*}
Hence,
$$F(i,j-i)\cdot \tr(\phi_{j-i}\phi_i\psi_j)=\frac{\theta'_i(0)}{\theta_i(0)}+
\frac{\theta'_{j-i}(0)}{\theta_{j-i}(0)}-2\pi \ii n.$$
Thus, recalling the definition $F(0,i)$, we can rewrite \eqref{C-F-tr-eq} as
$$2(C(0,j-i)_M-C(0,j)_M-C(0,-i)_M)=2\frac{\theta'_i(0)}{\theta_i(0)}-2\frac{\theta'_j(0)}{\theta_j(0)}+
\frac{\theta'_{j-i}(0)}{\theta_{j-i}(0)}-\frac{\theta'_{i-j}(0)}{\theta_{i-j}(0)}.$$
Finally, using property (3) of theta-functions, we get for $\alpha\neq 0$,
$$\frac{\theta'_{-\alpha}(0)}{\theta_{-\alpha}(0)}=2\pi \ii n - \frac{\theta'_\alpha(0)}{\theta_\alpha(0)}.$$
Hence, we can rewrite the above formula as
\begin{align*}
&C(0,j-i)_M-C(0,j)_M-C(0,-i)_M=\frac{\theta'_{j-i}(0)}{\theta_{j-i}(0)}
-\frac{\theta'_{-i}(0)}{\theta_{-i}(0)}-\frac{\theta'_j(0)}{\theta_j(0)}+\pi \ii n\\
&=F(0,j-i)-F(0,j)-F(0,-i).
\end{align*}
Combining this with \eqref{C-M-F-eq}, we see using \eqref{H-inv-proj-Pois-eq} that
our Poisson bracket on the projective space coincides with the one induced by the Sklyanin bracket,
which finishes the proof.

\begin{remark}
It is believed that if $\xi$ is a stable vector bundle of rank $k$ and degree $n$ then the Poisson structure on the projective
space $N_\xi$ coincides with the one obtained from
the quadratic Poisson bracket $q_{n,k}$. However, the computation is much more complicated. We leave it for the future work.
\end{remark}

\begin{remark}
Using Corollary \ref{leaves}, we can classify the symplectic leaves of $N_\xi$ completely. Such a classification was first claimed in a seminal paper of Feigin and Odesskii (Theorem 1 \cite{FO95}). However, they only claimed certain set theoretical bijection and the argument seems to be incomplete. We will give a proof of this classification in a forthcoming paper \cite{HLP}.
\end{remark}

\end{document}